\documentclass[11pt]{article}
\usepackage[margin=2.5cm]{geometry}

 \usepackage[pdftex]{graphicx}
\usepackage{amsmath}
\usepackage{algorithm}
\usepackage{algorithmicx}
\usepackage{multicol}
\usepackage{amsfonts}
\usepackage{amssymb}
\usepackage{amsthm}
\usepackage{color}
\usepackage{enumitem}
\usepackage{subfigure}
\usepackage[hidelinks,bookmarks=false]{hyperref}
\usepackage{ulem}
\usepackage{cite}
\usepackage{authblk}
\newcommand{\set}[1]{\left\{#1\right\}}
\newcommand{\iprods}[1]{\langle{#1}\rangle}
\newtheorem{definition}{Definition}[section]
\newtheorem{lemma}{Lemma}[section]
\newtheorem{proposition}{Proposition}[section]
\newcommand{\norms}[1]{\Vert#1\Vert}
\newtheorem{theorem}{Theorem}[section]

\newcommand{\xb}{\boldsymbol{x}}
\newcommand{\yb}{\boldsymbol{y}}
\newcommand{\wb}{\boldsymbol{w}}
\newcommand{\db}{\boldsymbol{d}}
\newcommand{\Vb}{\boldsymbol{V}}
\newcommand{\rb}{\boldsymbol{r}}
\newcommand{\ub}{\boldsymbol{u}}

\newcommand{\Gb}{\boldsymbol{G}}

\newcommand{\beforesec}{\vspace{-1ex}}
\newcommand{\aftersec}{\vspace{-0.5ex}}
\newcommand{\beforesubsec}{\vspace{-2ex}}
\newcommand{\aftersubsec}{\vspace{-0.5ex}}

\title{{A Globally Convergent Penalty-Based Gauss-Newton 
 Algorithm with Applications}}
\author[1]{Ilyes~Mezghani}
\author[2]{Quoc~Tran-Dinh}
\author[3]{Ion~Necoara}
\author[4]{Anthony~Papavasiliou}
\affil[1]{{\small Corresponding Author. Center for Operations Research and Econometrics, Universit\'e Catholique de Louvain. Voie du Roman Pays 34, 1348 Louvain-la-Neuve, Belgium. E-mail: \texttt{ilyes.mezghani@uclouvain.be}}}
\affil[2]{{\small Department of Statistics and Operations Research,
The University of North Carolina at Chapel Hill. 333 Hanes Hall,    
Chapel Hill, NC 27599, USA. E-mail: \texttt{quoctd@email.unc.edu}}}
\affil[3]{{\small Politehnica University of Bucharest,
Department of Automatic Control and Systems Engineering.
Splaiul Independentei nr. 313, sector 6,
060042 Bucharest,
Romania. E-mail: \texttt{ion.necoara@acse.pub.ro}}}
\affil[4]{{\small Center for Operations Research and Econometrics, Universit\'e Catholique de Louvain. Voie du Roman Pays 34, 1348 Louvain-la-Neuve, Belgium. E-mail: \texttt{anthony.papavasiliou@uclouvain.be}}}
\date{}
\begin{document}
\maketitle
\begin{abstract}
We propose a globally convergent Gauss-Newton algorithm for finding a local optimal solution of a non-convex and possibly non-smooth optimization problem.  
The algorithm that we present is based on a Gauss-Newton-type iteration for the non-smooth penalized formulation of the original problem. 
We establish a global convergence rate for this scheme from any initial point to a stationary point of the problem while using an exact penalty formulation. 
Under some more restrictive conditions we also derive local quadratic convergence for this scheme.  
We apply our proposed algorithm to solve the {Alternating Current} optimal power flow problem on meshed electricity networks, which is a fundamental application in power systems engineering. 
We verify the performance of the proposed method by {showing} comparable behavior with IPOPT, a well-established solver. We perform our validation on several representative instances of the optimal power flow problem, which are sourced from the MATPOWER library. 
\end{abstract}


\newpage
\beforesec
\section{Introduction}\label{sec:intro}

\paragraph{Statement of the problem.}
In this paper we are interested in solving the following optimization problem with non-convex constraints:
\begin{equation}
\label{eq:constr_prob}
\min_{\boldsymbol{x}  \in \mathbb{R}^d} f(\boldsymbol{x})~~\text{s.t.}~~\Psi(\boldsymbol{x}) = 0, ~ \boldsymbol{x} \in \Omega.
\end{equation}

For this problem we assume that the objective function $f$ is convex and differentiable, $\Omega$ is a compact convex set, and the non-convexity enters in \eqref{eq:constr_prob} through the non-linear equality constraints  $\Psi(\boldsymbol{x}) = 0$, defined by $\Psi : \mathbb{R}^d\to \mathbb{R}^n$. 

Note that in the literature we can find many efficient algorithms that are able to minimize  non-convex objective functions, but with convex constraints, see e.g., \cite{CheSun:19,BolSab:14, PatNec:15,Nocedal2006}. In this paper, we treat a more general optimization model, where the non-convexity enters into the optimization problem through the constraints. 

It is well-known that  optimization problems with non-convex constraints are more difficult to solve than convex constrained problems \cite{Nocedal2006}.  
For the non-convex problem \eqref{eq:constr_prob}, classical non-convex optimization algorithms such as interior-point, augmented Lagrangian, penalty, Gauss-Newton, and sequential quadratic programming methods can only aim at finding a stationary point (i.e. a point that satisfies the first-order optimality conditions), which is a candidate for a local minimum \cite{Nocedal2006, CarGou:11}. 
{Nevertheless, convergence guarantees of these methods rely on the twice continuous differentiability of the underlying functionals, including the representation of $\Omega$.
Our method also approximates a stationary point, but it allows $f$ to include nonsmooth convex terms and $\Omega$ to be a general convex set (see Subsection~\ref{subsec:bmi_opt}).  
}

For an iterative method to identify a stationary point that is a local minimum, but not a saddle-point, more sophisticated techniques are required, such as cubic regularization \cite{Nesterov2006} or random noise gradient \cite{DauPas:14}.   
However, it is still unclear how to efficiently implement these methods for solving large-scale optimization problems with non-convex constraints. One of the most efficient and well-established non-linear solvers for finding stationary points is IPOPT \cite{Wachter2006}, which relies on a primal-dual interior-point method combined with other advanced techniques. 
We emphasize that  this classical method is only guaranteed to converge to a stationary point, and often requires a strategy such as line-search, filter, or trust-region to achieve global convergence (i.e. the method is still convergent from a starting point that is far from the targeted stationary point) under certain restrictive assumptions.  
{These methods often assume that the underlying functionals in \eqref{eq:constr_prob} are twice continuously differentiable, including $f$ and  $\Omega$.
In addition, a linesearch procedure may require that the gradient of the underlying merit function is Lipschitz continuous to guarantee global convergence, see, e.g., \cite[Theorem 3.2.]{Nocedal2006}.
}

\paragraph{A Gauss-Newton algorithm for constrained non-convex optimization.}
Recently there has been a revived interest in the design and analysis of algorithms for solving optimization problems involving non-convex constraints, in particular in  engineering and machine learning \cite{boob2019proximal, bolte2016majorization, cartis2014complexity, curtis2018complexity, TranDinh2012}. 
The main trend is in solving large-scale problems by exploiting special structures/properties of the problem model  and data towards the design of  simple schemes (e.g.,  solving a tractable convex subproblem at each iteration), while producing reasonable approximate solutions efficiently \cite{drusvyatskiy2019efficiency, Lewis2008, Nesterov2007g}. 

Following this trend, in this paper we also devise a  provable convergent  Gauss-Newton (GN)-type algorithm for solving the non-convex optimization problem \eqref{eq:constr_prob}.  
{The idea of the GN method studied in this paper was proposed in \cite{Lewis2008} for minimizing a compositional model of the form $\phi(F(x))$, where $\phi$ is possibly nonsmooth.
This method targets a different problem class compared to standard GN or Levenberg--Marquardt methods for nonlinear least-squares problems. 
The main idea is to replace the non-Lipschitz continuous least-squares function $\phi(\cdot) = \frac{1}{2}\Vert \cdot\Vert_2^2$ in these methods by a given convex and Lipschitz continuous function $\phi(\cdot)$ (but possibly nonsmooth).
Nesterov contributed a thorough investigation on convergence guarantees of this method in \cite{Nesterov2007g} when $\phi$ is a given norm.
This was extended to a more general model that can cover exact penalty methods in a technical report \cite{Tran-Dinh2011}.
Very recently, \cite{drusvyatskiy2019efficiency} revisited this method under a Moreau's envelope perspective, but only on global convergence guarantees.
}

Our algorithm converges globally to a stationary point of the problem in the sense that, starting from any initial point within a given level set, the algorithm converges to a stationary point.
In addition, the proposed approach is different from standard GN methods in the literature \cite{Nocedal2006, Deuflhard2006} due to the use of a non-smooth penalty instead of a classical quadratic penalty term.  
This allows our algorithm to converge globally \cite{Nes:07}. 
Hence, we refer to this algorithm as a \textit{global} GN scheme.  

The main idea of our method is to keep the convex sub-structure of the original problem  unchanged and to convexify the non-convex part by exploiting penalty theory and the GN framework.   
Hence, in contrast to IPOPT (solving a linear system of the barrier problem to obtain a Newton search direction), each iteration of our algorithm requires finding the solution of a strongly convex subproblem, which can be  efficiently solved by many existing convex solvers. 
Under some more restrictive conditions, we also derive a local quadratic convergence rate for our GN method.

\paragraph{The optimal power flow problem.} 
We apply our proposed optimization algorithm to  the alternating current optimal power flow (AC-OPF) problem \cite{FraSte:12}, which lies at the heart of short-term power system operations \cite{Cain2012}. We show that this problem can be posed in the framework of non-convex optimization with the particular structure on the constraints as in  \eqref{eq:constr_prob}. The optimal power flow (OPF) problem \cite{Carpentier1962}  consists in finding an optimal operating point of a power system while minimizing a certain objective (typically power generation cost), subject to the Kirchhoff's power flow equations and  various network and control operating limits.   

In recent years, there has been a great body of literature that has focused on convex relaxations of the AC-OPF problem, including semidefinite programming relaxations \cite{LavLow:12,Kocuk2016SDP}, conic relaxations \cite{Jabr2008,Gan2015,Kocuk2016SOCP}, and quadratic relaxations \cite{Coffrin2016}. These works have established conditions under which these relaxations are exact, and understanding cases in which this is not so \cite{Molzahn2013}.  However, 
when these relaxations are inexact, the resulting dispatch is possibly non-implementable. Therefore, our interest in the present paper is to tackle directly this problem as a non-convex optimization problem with non-linear equality constraints. 



\paragraph{Contributions.}
The main contributions of the paper are the following:
\begin{itemize}
\item[(i)] 
We propose a new GN algorithm for solving a general  class of optimization problems with non-convex constraints. 
We utilize an exact non-smooth penalty reformulation of the original problem and suggest a GN scheme to solve this penalized problem, where the subproblem in this scheme is a strongly convex program, which can be efficiently solved by several recent and highly efficient third-party convex optimization solvers. 
{
Since our proposed approach preserves convexity in $f$ and $\Omega$ of \eqref{eq:constr_prob} in the subproblem, our method can solve a broader class of problems with theoretical guarantee than classical IP or SQP methods. 
In particular, our method can solve problem instances of \eqref{eq:constr_prob} with nonsmooth convex objective terms or semidefinite cone constraints, see Subsection~\ref{subsec:bmi_opt} for a concrete example.
}

\item[(ii)] 
We establish the best-known global  convergence rate (i.e., convergence from any starting point in a given sublevel set) for our method to a stationary point, which is a candidate for a local optimal solution.
Moreover, under some more restrictive conditions, we also derive a local quadratic convergence for our scheme.  

%


\item[(iii)] We apply our proposed algorithm to the quadratic formulation of AC-OPF. We show that the newly developed algorithm can be implemented efficiently on AC-OPF problems and test it on several numerical examples from the well-known MATPOWER library \cite{ZimMur:11}. 
We often observe comparable performance to the well-established and widely-used IPOPT solver. {Also, since our approach relies on an iterative scheme, using the previous iterate to accelerate the computation of the next iterate can significantly improve performance: this is what we refer to as warm-start. Even though most solvers do not possess warm-start capabilities for this type of problems, we show how warm-start can significantly improve the performances of the proposed method. Section 3 provides details regarding performance and potential improvements.}
\end{itemize}

We emphasize that, as opposed to the classical GN approach, our proposed algorithm relies on the $\ell_1$-norm penalty, which typically has a better condition number than the quadratic penalty, as discussed in \cite{Nesterov2007g}.
More generally, if we replace this $\ell_1$-penalty with any other exact and Lipschitz continuous penalty function, then our theoretical results still hold.
Unlike certain sophisticated methods such as Interior-Point Methods  (IPMs) and Sequential Quadratic Programming (SQP) schemes, our GN method is simple to implement. Its workhorse is the solution of a strongly convex problem. 
As a result, its efficiency depends on the efficiency of a third-party solver for this problem as well as the benefit of warm-start strategies. 
Our main motivation is to exploit recent advances in large-scale convex optimization in order to create a flexible algorithm that can reuse this resource.




\paragraph{Content.}
The paper is organized as follows.  In Section~\ref{sec:GN_alg}, we introduce our Gauss-Newton algorithm and analyze its convergence properties. 
In Section~\ref{sec:modeling}, we present the AC-OPF problem, its quadratic reformulation and test our Gauss-Newton algorithm on several representative  MATPOWER test cases.  

\aftersec

\beforesec 
\section{A Gauss-Newton Algorithm for Non-Convex Optimization}\label{sec:GN_alg}
\aftersec

In this section, we present the main assumptions for the non-convex optimization problem \eqref{eq:constr_prob}, propose an exact penalty reformulation, and solve it  using a Gauss-Newton-type algorithm.  We further characterize the global and local  convergence rates of our algorithm. 

\beforesubsec
\subsection{Exact penalty approach for constrained non-convex programming}
For the non-convex optimization problem \eqref{eq:constr_prob} we assume  that the objective function $f$ is convex and differentiable and $\Omega$ is a compact convex set. 
Note that our method developed in the sequel can also be extended non-smooth convex function $f$ or smooth non-convex function $f$ whose gradient is Lipschitz continuous, but we make this assumption for simplicity of presentation.
Furthermore,  the non-convexity enters into the optimization problem through the non-linear equality constraints  $\Psi(\boldsymbol{x}) = 0$ defined by $\Psi : \mathbb{R}^d\to \mathbb{R}^n$.  We assume that $\Psi$ is differentiable and its Jacobian  $\Psi'$ is Lipschitz continuous, i.e. there exists $L_{\Psi} >0$ such that:
\begin{equation*}
\Vert \Psi'(\xb) - \Psi'(\hat{\xb})\Vert \leq L_{\Psi}\Vert \xb - \hat{\xb}\Vert \quad  \forall  \xb, \hat{\xb} \in \Omega, 
\end{equation*}
where $\Vert\cdot\Vert$ is the $\ell_2$-norm.  Further, let $\mathcal{N}_{\Omega}$ denote the normal cone of the convex set $\Omega$:
\begin{equation*} 
\mathcal{N}_{\Omega}(\xb) := \begin{cases} \left\{ \wb\in\mathbb{R}^d~\mid~ \wb^{\top}(\yb - \xb) \geq 0,~ \forall\yb\in\Omega \right\}, \ \text{if}~ \xb\in\Omega\\
\emptyset, \quad \text{otherwise}.
\end{cases}
\end{equation*}
Since problem~\eqref{eq:constr_prob} is non-convex, our goal is to search for a stationary point 
of this optimization problem  that is a candidate for a  local optimum in the following sense.
\begin{definition}[\cite{Nocedal2006}(Theorem 12.9)]\label{de:kkt_and_stationary}
A point $(\xb^{*}, \yb^{*})$ is said to be a KKT point of \eqref{eq:constr_prob} if it satisfies the following conditions:
\begin{equation}\label{eq:kkt_cond}
-\nabla f(\xb^{*}) - \Psi'(\xb^{*})\yb^{*} \in \mathcal{N}_{\Omega}(\xb^{*}), \quad \xb^{\ast}  \in \Omega, \quad \text{and}\quad \Psi(\xb^{*}) = 0. 
\end{equation}
Here, $\xb^{*}$ is called a stationary point of \eqref{eq:constr_prob}, and $\yb^{*}$ is the corresponding multiplier. 
Let $\mathcal{S}^{*}$ denote the set of these stationary points.
\end{definition}

\noindent Since $\Omega$ is compact, and $\Psi$ and $f$ are continuous, by the well-known Weierstrass theorem, we have:

\begin{proposition}\label{pro:global_solution}
If $\Omega\cap\set{ \xb \mid \Psi(\xb) = 0}\neq\emptyset$, then  \eqref{eq:constr_prob} has global optimal solutions.
\end{proposition}

\beforesubsec
\subsection{Exact penalized formulation}
Associated with \eqref{eq:constr_prob}, we consider its exact penalty form \cite[Chapt. 17.3]{Nocedal2006}:
\begin{equation}\label{eq:NLP}
\min_{\xb \in\Omega}  \Big\{ F(\xb) := f(\xb) + \beta \vert\Psi(\xb)\vert \Big\},
\end{equation}
where $\beta > 0$ is a penalty parameter, and $\vert\cdot\vert$ is the $\ell_1$-norm. Two reasons for choosing an exact (non-smooth) penalty are as follows. 
First, for a certain finite choice of the  parameter $\beta$, a single minimization in $x$ of \eqref{eq:NLP} can yield an exact solution of the original problem \eqref{eq:constr_prob}. 
Second, it does not square the condition number of $\Psi$ as in the case of quadratic penalty methods, thus making our algorithm presented below more robust to ill-conditioning of the non-convex  constraints. 
Now, we summarize the relationship between stationary points of \eqref{eq:constr_prob} and of its penalty form \eqref{eq:NLP}. For this, let us define the directional derivative:
\begin{equation}\label{eq:dir_deriv}
DF(\xb^{*})[\db] := \nabla f(\xb^{*})^{\top}\db + \beta \xi(\xb^{*})^{\top} \Psi'(\xb^{*})^{\top} \db,
\end{equation}  
where $\xi(\xb^{*}) \in \partial\vert \Psi(\xb^{*})\vert$ is one subgradient of $\vert\cdot\vert$ at $\Psi(\xb^{*})$, and $\partial{\vert\cdot\vert}$ denotes the subdifferential of $\vert\cdot\vert$, see \cite{Nes:07}.  
Recall that the necessary optimality condition of \eqref{eq:NLP} is
\begin{equation*}\label{eq:FONC}
0 \in \nabla f(\xb^{*}) + \beta \Psi'(\xb^{*})  \partial\vert\Psi(\xb^{*})\vert + \mathcal{N}_{\Omega}(\xb^{*}).
\end{equation*}
Then, this condition can be expressed equivalently as
\begin{equation}\label{eq:FONC_new}
DF(\xb^{*})[\db] \geq 0, ~~\forall \db \in \mathcal{F}_{\Omega}(\xb^{*}),
\end{equation}
where $\mathcal{F}_{\Omega}(\xb)$ is the  set of feasible directions to $\Omega$ at $\xb$:
\begin{equation}\label{eq:F_Omega}
\mathcal{F}_{\Omega}(\xb) := \left\{ \db\in \mathbb{R}^d \mid \db = t(\yb - \xb), ~\forall \yb\in\Omega, ~t\geq 0 \right\}.
\end{equation}
Any point $\xb^{*}$ satisfying \eqref{eq:FONC_new} is called a stationary point of the penalized problem \eqref{eq:NLP}. Stationary points are candidates for local minima, local maxima, and saddle-points.
If, in addition, $\xb^{*}$ is feasible to \eqref{eq:constr_prob}, 
then we say that $\xb^{*}$ is a feasible stationary point.
Otherwise, we say that $\xb^{*}$ is an infeasible stationary point.  Proposition~\ref{pro:penalty_stationary} shows the relation between  \eqref{eq:constr_prob} and \eqref{eq:NLP}.

\begin{proposition}[\cite{Nocedal2006}, (Theorem 17.4.)]\label{pro:penalty_stationary}
Suppose that $\xb^{*}$ is a feasible stationary point of \eqref{eq:NLP} for $\beta$ sufficiently large.
Then, $\xb^{*}$ is also stationary point of the original problem~\eqref{eq:constr_prob}.
\end{proposition}

\noindent 
{Proposition~\ref{pro:penalty_stationary} requires $\xb^{*}$ to be feasible for \eqref{eq:constr_prob}.
When the feasible set $\Omega\cap\set{ \xb \mid \Psi(\xb) = 0}\neq\emptyset$ of \eqref{eq:constr_prob} is nonempty and bounded, according to \cite[Proposition 2]{di1994exact}, if \eqref{eq:constr_prob} satisfies the extended Mangarasian-Fromovitz constrained qualification condition (see \cite[Proposition 2]{di1994exact} for concrete definition), then there exists $\beta_{*} > 0$ such that for any $\beta > \beta_{*}$, every global or local solution of the penalized problem \eqref{eq:NLP} is also a global or local optimal solution of \eqref{eq:constr_prob}, respectively.
By  \cite[Proposition 3]{di1994exact}, $\beta$ needs to be chosen such that $\beta >  \beta_{*} := \|\yb^{*}\|_\infty$, where $\yb^{*}$ is any optimal Lagrange multiplier of \eqref{eq:constr_prob}. 
We will discuss in detail the choice of $\beta$ in  the sections below.}

\beforesubsec
\subsection{Global Gauss-Newton method}
We first develop our GN algorithm. 
Then, we investigate its global  convergence rate.

\subsubsection{The derivation of the Gauss-Newton scheme and the full algorithm}
Our GN method aims at solving the penalized problem \eqref{eq:NLP} using the following convex subproblem:
\begin{equation}\label{eq:GN_dir}
{\!\!\!\!\!\!\!\!\!\!}\begin{array}{ll}
&\displaystyle\min_{\xb \in \Omega} \Big\{ \mathcal{Q}_L(\xb ; \xb^k) := f(\xb)  +  \beta \vert \Psi(\xb^k) + \Psi'(\xb^k)(\xb - \xb^k)\vert  + \tfrac{L}{2} \Vert \xb - \xb^k \Vert^2 \Big\},
\end{array}{\!\!\!\!}
\end{equation}
where $\xb^k$ is a given point in $\Omega$ for linearization, $\Psi'(\cdot)$ is the Jacobian of $\Psi$, and $L > 0$ is a regularization parameter. 

Note that our subproblem \eqref{eq:GN_dir} differs  from those used  in classical penalty methods \cite{Nocedal2006}, since we linearize the constraints and we also add a regularization term. 
Thus, the objective function of \eqref{eq:GN_dir} is strongly convex.
Hence, if $\Omega$ is nonempty {and even if the problem is non-differentiable}, this problem admits a unique optimal solution, and can be solved efficiently by several convex methods and solvers. 
{For instance, alternating direction methods of multipliers (ADMM) \cite{Boyd2011} and primal-dual schemes \cite{Chambolle2011} can be efficient for solving \eqref{eq:GN_dir}.
Note that the convergence guarantees of ADMM and primal-dual schemes often depends on the distance between the initial point $\xb^{k,0}$ of the algorithm and the exact optimal solution of $\bar{\xb}^{k+1}$ of \eqref{eq:GN_dir}, see, e.g, \cite[Theorem 2]{Chambolle2011}.
Hence, if we warm-start $\xb^{k,0}$ at the previous approximate solution $\xb^k$ obtained at the $(k-1)$-th iteration, then the distance $\Vert \xb^0 - \bar{\xb}^{k+1}\Vert$ is small. This allows the algorithm to converge faster to a desired approximate solution $\xb^{k+1}$ of \eqref{eq:GN_dir} at the $k$-th iteration.
}

Let us define
\begin{equation}\label{eq:subprob1_a}
{\!\!\!\!}\begin{array}{ll}
&\Vb_L(\xb^k) :=  \text{arg}\!\displaystyle\min_{x\in\Omega}\Big\{ \mathcal{Q}_L(\xb; \xb^k) := f(\xb)  +  \beta \vert \Psi(\xb^k) + \Psi'(\xb^k)(\xb - \xb^k)\vert + \frac{L}{2} \Vert \xb - \xb^k \Vert^2 \Big\}.{\!\!\!\!}
\end{array}
\end{equation}
The necessary and sufficient optimality condition for subproblem \eqref{eq:GN_dir} becomes 
\begin{equation}\label{eq:subprob1_b}
{\!\!\!\!\!\!\!}\begin{array}{ll}
&\left[\nabla{f}(\Vb_L(\xb^k)) + L(\Vb_L(\xb^k) - \xb^k) + \beta \Psi'(\xb^k)  \xi(\xb^k)\right]^{\top}(\hat{\xb} - \Vb_L(\xb^k)) \geq 0, ~~\forall \hat{\xb}\in\Omega,
\end{array}{\!\!\!\!}
\end{equation}
where $\xi(\xb^k) \in \partial\vert \Psi(\xb^k) + \Psi'(\xb^k)(\Vb_L(\xb^k) - \xb^k)\vert$. Given $\Vb_L(\xb^k)$, we define the following quantities:
\begin{align}\label{eq:gradient_mapping}
&\Gb_L(\xb^k) := L(\xb^k - \Vb_L(\xb^k)),~\db_L(\xb^k) := \Vb_L(\xb^k) - \xb^k,~~\text{and}~~r_L(\xb^k) := \Vert \db_L(\xb^k)\Vert. 
\end{align}
Then, $\Gb_L(\cdot)$ can be considered as a gradient mapping of $F$ in \eqref{eq:NLP} \cite{Nes:07}, and $\db_L(\xb^k)$ is a search direction for Algorithm~\ref{alg:A1}. 
{As we will see later in Lemma~\ref{le:well_define}}, $L$ should be chosen such that $0 < L \leq \beta L_{\Psi}$. Now, using the subproblem \eqref{eq:GN_dir} as a main component, we describe our GN scheme in Algorithm \ref{alg:A1}.

\begin{algorithm}\caption{{\!}(\textit{The Basic Gauss-Newton Algorithm}){\!\!\!\!}}\label{alg:A1}
\begin{normalsize}
\begin{algorithmic}[1]
	\State {\hskip0ex}\textbf{Initialization:}
	Choose $\xb^0 \in \Omega$ and a penalty parameter $\beta > 0$ sufficiently large (ideally, $\beta > \Vert\yb^{*}\Vert_\infty$).
	\State Choose a lower bound $L_{\min} \in (0,  \beta L_{\Psi}]$. 
	\vspace{1ex}
	\State \textbf{For $k := 0$ to $k_{\max}$ perform}
		\vspace{1ex}
		\State{\hskip2ex}\label{step:A_1_GN_step}
		Find $L_k \in [L_{\min}, \beta L_{\Psi}]$ such that $F(\Vb_{L_k} (\xb^k)) \leq \mathcal{Q}_{L_k} (\Vb_{L_k}(\xb^k); \xb^k)$ (see Lemma~\ref{le:well_define}).
                 \vspace{1ex}
                 \State{\hskip2ex}\label{step:A1_x_update}
                 Update $\xb^{k+1} := \Vb_{L_k} (\xb^k)$. 
                 \State{\hskip2ex}\label{step:A1_beta_update}
                 Update $\beta$  if necessary.
	\State\textbf{End~for}
\end{algorithmic}
\end{normalsize}
\end{algorithm}

The main step of Algorithm~\ref{alg:A1} is the solution of the convex subproblem  \eqref{eq:GN_dir} at Step~\ref{step:A_1_GN_step}.
As mentioned, this problem is strongly convex, and can be solved by several methods that converge linearly.
If we choose $L_k \equiv L \geq \beta L_{\Psi}$, then we do not need to perform a line-search on $L$ at Step~\ref{step:A_1_GN_step}, and {only need to solve \eqref{eq:GN_dir} once per iteration}.
However, {$L_{\Psi}$ may not be known} or if it is known, the global upper bound $\beta L_{\Psi}$ may be too conservative, i.e. it does not take into account the local structures of non-linear functions in \eqref{eq:NLP}.
{Therefore, following the algorithm in \cite{Nes:07}, we propose performing a line-search in order to find an appropriate $L_k$.
If we perform a line-search by doubling $L_k$ at each step starting from $L_{\min}$, (i.e., $L_k \rightarrow 2L_k$), then after $i_k$ line-search steps, we have $L_k = 2^{i_k}L_{\min}$, and  the number of line-search iterations $i_k$ is at most $\lfloor\log_2(\beta L_{\Psi}/L_{\min})\rfloor + 1$.
Note that it is rather straightforward to estimate $L_{\min}$. For example, we can set $L_{\min} := \frac{c\beta\Vert\Psi'(\hat{\xb}^0) - \Psi'(\xb^0)\Vert}{\Vert\hat{\xb}^0 - \xb^0\Vert} \leq \beta L_{\Psi}$ for some $\hat{\xb}^0\neq\xb^0$ and $c \in (0, 1]$.} 
The penalty parameter $\beta$ can be fixed or can be updated gradually using a run-and-inspect strategy, see next section.  

\subsubsection{Global convergence analysis}
We first  summarize some properties of Algorithm~\ref{alg:A1} when the penalty parameter $\beta$ {is fixed at a given positive value for all iterations}.

\begin{lemma}\label{le:stationary_point}
Let $\Vb_L$ be defined by \eqref{eq:subprob1_a}, and $\Gb_L$, $\db_L$, and $r_L$ be defined by \eqref{eq:gradient_mapping}.
Then the following statements hold:
\begin{itemize}
\item[$\mathrm{(a)}$] 
If $\Vb_{L_k}(\xb^k) = \xb^k$,  then $\xb^k$ is a stationary point of \eqref{eq:NLP}.

\item[$\mathrm{(b)}$] 
The norm $\norms{\Gb_{L_k}(\xb^k)}$ is nondecreasing in ${L_k}$, and $r_{L_k}(\xb^k)$ is nonincreasing in ${L_k}$.
Moreover, we have
\begin{equation}\label{eq:f_f_rho}
F(\xb^k) - \mathcal{Q}_{L_k}(\Vb_{L_k}(\xb^k); \xb^k) \geq \frac{{L_k}}{2}r^2_{L_k}(\xb^k).
\end{equation}
\item[$\mathrm{(c)}$] 
If $\Psi'(\cdot)$ is Lipschitz continuous with the Lipschitz constant $L_{\Psi}$, then, for any $x\in\Omega$, we have
\begin{equation} \label{eq:key_est1}
{\!\!\!\!\!\!\!\!\!\!\!}\begin{array}{ll}
&F(\xb^k) - F(\Vb_{L_k}(\xb^k)) \geq  \frac{(2{L_k} - \beta L_{\Psi})}{2}r^2_{L_k}(\xb^k) = \frac{(2{L_k} - \beta  L_{\Psi})}{2 L_k^2}\norms{\Gb_{L_k}(\xb^k)}^2. \vspace{1ex}\\
&DF(\xb^k)[\db_{L_k}(\xb^k)] \leq - {L_k} r^2_{L_k}(\xb^k) = -\frac{1}{{L_k}}\norms{\Gb_{L_k}(\xb^k)}^2.
\end{array}{\!\!\!\!\!\!\!\!}
\end{equation}
\end{itemize}
\end{lemma}

\begin{proof}
(a)~Substituting $\Vb_{L_k}(\xb^k) = \xb^k$ into \eqref{eq:subprob1_b}, we again obtain the optimality condition \eqref{eq:FONC_new}.
This shows that $\xb^k$ is a stationary point of \eqref{eq:NLP}.

(b)~Since the function $q(t, \xb) := f(\xb) + \beta \vert \Psi(\xb^k) + \Psi'(\xb^k)(\xb - \xb^k)\vert + \frac{1}{2t} \Vert \xb - \xb^k \Vert^2 $ is convex in two variables $\xb$ and $t$, 
we have that $\eta(t) := \min_{\xb\in \Omega}q(t, \xb)$ is still convex. 
It is easy to show that $\eta'(t) = -\frac{1}{2t^2}\Vert \Vb_{1/t}(\xb^k) - \xb^k\Vert^2 = -\frac{1}{2t^2}\Vert \db_{1/t}(\xb^k)\Vert^2 = \frac{1}{2}\Vert \Gb_{1/t}(\xb)\Vert^2$.
Since $\eta(t)$ is convex, $\eta'(t)$ is nondecreasing in $t$. 
This implies that $\Vert \Gb_{1/t}(\xb^k)\Vert$ is nonincreasing in $t$. 
Thus $\Vert \Gb_{L}(\xb^k)\Vert$ is nondecreasing in $L$ and $r_L(\xb^k) := \Vert \db_{L}(\xb^k)\Vert$ is nonincreasing in $L$. To prove \eqref{eq:f_f_rho}, note that the convexity of $\eta$ implies that
\begin{equation}\label{eq:lm21_eq1}
F(\xb^k) = \eta(0) \geq \eta(t) + \eta'(t)(0-t) = \eta(t) + \frac{1}{2t}\rb^2_{1/t}(\xb^k).
\end{equation} 
On the other hand, $\mathcal{Q}_{L}(\Vb_L(\xb^k); \xb^k) = \eta(1/L)$. 
Substituting this relation into \eqref{eq:lm21_eq1}, we obtain \eqref{eq:f_f_rho}. 

(c)~Let use define $\Vb_k := \Vb_{L_k}(\xb^k)$.
From the optimality condition \eqref{eq:subprob1_b}, for any $\xb\in\Omega$, we have
\begin{equation*} 
{\!\!\!\!}\begin{array}{ll}
\left[\nabla{f}(\Vb_k) + {L_k}(\Vb_k - \xb^k) + \beta \Psi'(\xb^k) \xi(\xb^k)\right]^{\top}(\xb - \Vb_k) \geq 0,
\end{array}{\!\!\!\!}
\end{equation*}
where $\xi(\xb^k) \in \partial\vert  \Psi(\xb^k) + \Psi'(\xb^k)(\Vb_k - \xb^k)\vert$.
Substituting $\xb  = \xb^k$ into this condition, we have
\begin{equation}\label{eq:proof3} 
\nabla{f}(\Vb_k)^{\top}(\xb^k  - \Vb_k ) + \beta \xi(\xb^k)^{\top} \Psi'(\xb^k)^{\top} (\xb^k - \Vb_k) \geq {L_k}  \Vert \Vb_k - \xb^k\Vert^2.
\end{equation}
{Since $f$ is convex, we have:
\begin{equation*}
\begin{array}{ll}
f(\xb^k) &\geq f(\Vb_k) + \nabla{f}(\Vb_k)^{\top}(\xb^k - \Vb_k). \vspace{1ex}
\end{array}
\end{equation*}
By exploiting the convexity of $\vert\cdot\vert$ at point $\Psi(\xb^k)  + \Psi'(\xb^k)(\Vb_k - \xb^k)$, we have:
\begin{equation*}
\begin{array}{ll}
\vert\Psi(\xb^k)\vert &\geq \vert \Psi(\xb^k)  + \Psi'(\xb^k)(\Vb_k - \xb^k)\vert  +   \xi(\xb^k)^{\top}\Psi'(\xb^k)^{\top}(\xb^k -\Vb_k).
\end{array}
\end{equation*}}
Since $\Psi'$ is Lipschitz continuous, we also have 
\begin{equation*}
{\!\!\!\!\!\!\!\!}\begin{array}{ll}
\vert\Psi(\Vb_k)\vert &\leq \vert \Psi(\xb^k)  + \Psi'(\xb^k)(\Vb_k - \xb^k)\vert  +  \vert \Psi(\Vb_k) - \Psi(\xb^k)  + \Psi'(\xb^k)(\Vb_k - \xb^k)\vert \vspace{1ex}\\
&\leq \vert  \Psi(\xb^k)  + \Psi'(\xb^k)(\Vb_k - \xb^k)\vert   +  \frac{L_{\Psi}}{2}\Vert \Vb_k - \xb^k\Vert^2.
\end{array}
\end{equation*}
Combining these three bounds, we can show that
\begin{equation*}
{\!\!\!\!\!\!}\begin{array}{ll}
f(\xb^k) + \beta\vert\Psi(\xb^k)\vert &\geq f(\Vb_k)  +  \beta\vert \Psi(\xb^k)  + \Psi'(\xb^k)(\Vb_k - \xb^k)\vert  +  {L_k} \Vert \Vb_k - \xb^k\Vert^2 \vspace{1ex}\\
&\geq f(\Vb_k) + \beta\vert\Psi(\Vb_k)\vert  +  {L_k} \Vert \Vb_k - \xb^k\Vert^2  - \frac{\beta L_{\Psi}}{2}\Vert \Vb_k - \xb^k\Vert^2,
\end{array}
\end{equation*}
which implies
\begin{equation*}
\begin{array}{ll}
F(\xb^k) &\geq F(\Vb_k) + \frac{(2{L_k}  - \beta L_{\Psi})}{2}\Vert \Vb_k - \xb^k\Vert^2.
\end{array}
\end{equation*}
Since $r^2_{L_k}(\xb^k) = \Vert \Vb_k - \xb^k\Vert^2 = \frac{1}{L_k^2}\Vert \Gb_{L_k}(\xb^k)\Vert^2$, we obtain the first inequality of \eqref{eq:key_est1} from the last inequality. Moreover, from \eqref{eq:dir_deriv} we have 
\begin{equation*} 
DF(\xb^k)[\db_{L_k}(\xb^k)] = \nabla{f}(\Vb_k)^{\top}(\Vb_k - \xb^k)  + \beta \xi(\xb^k)^{\top}\Psi'(\xb^k)^{\top}(\Vb_k - \xb^k).
\end{equation*}
Using \eqref{eq:proof3}, we can show that $DF(\xb^k)[\db_{L_k}(\xb^k)] \leq  - {L_k} \Vert \Vb_k - \xb^k\Vert^2$, which is the second inequality of  \eqref{eq:key_est1}.
\end{proof}

The proof of Statement $\mathrm{(a)}$ shows that if we can find $\xb^k$ such that $\norms{\Gb_{L_k}(\xb^k)}\leq\varepsilon$, then $\xb^k$ is an approximate stationary point of \eqref{eq:NLP} within the accuracy $\varepsilon$.
From statement $\mathrm{(b)}$, we can see that if the line-search condition $F(\Vb_{L_k}(\xb^k)) \leq \mathcal{Q}_{L_k}(\Vb_{L_k}(\xb^k);\xb^k)$ at Step~\ref{step:A_1_GN_step} holds, then $F(\Vb_{L_k}(\xb^k)) \leq F(\xb^k) - \frac{{L_k}}{2}r^2_{L_k}(\xb^k)$.
That is, the objective value $F(\xb^k)$ decreases at least by $\frac{{L_k}}{2}r^2_{L_k}(\xb^k)$ after the $k$-th iteration. We first claim that Algorithm~\ref{alg:A1} is well-defined.

\begin{lemma}\label{le:well_define}
Algorithm~\ref{alg:A1} is well-defined, i.e. step~\ref{step:A_1_GN_step} terminates after a finite number of iterations. That is, if $L \geq \beta L_{\Psi}$, then $F(\Vb_L(\xb^k)) \leq \mathcal{Q}_L(\Vb_L(\xb^k);\xb^k)$.
\end{lemma}

\begin{proof}
Since $\Psi'$ is $L_{\Psi}$-Lipschitz continuous, for any $\xb^k$ and $\Vb_L(\xb^k)$, we have 
\begin{equation*}
{\!\!\!\!\!}\begin{array}{ll}
\vert \Psi(\Vb_L(\xb^k))\vert {\!\!\!}&\leq \vert  \Psi(\xb^k) + \Psi'(\xb^k)(\Vb_L(\xb^k) - \xb^k)\vert  +  \Vert \Psi(\Vb_L(\xb^k)) - \Psi(\xb^k)   -  \Psi'(\xb^k)(\Vb_L(\xb^k) - \xb^k)\Vert \vspace{1ex}\\
&\leq \vert  \Psi(\xb^k) + \Psi'(\xb^k)(\Vb_L(\xb^k) - \xb^k)\vert  +  \frac{L_{\Psi}}{2}\Vert \Vb_L(\xb^k) - \xb^k\Vert^2.
\end{array}
\end{equation*}
Using the definition of $\mathcal{Q}_L(\Vb;\xb)$, we obtain
\begin{equation*}
\begin{array}{ll}
{\!\!\!}F(\Vb_L(\xb^k)) \leq \mathcal{Q}_L(\Vb_L(\xb^k)l\xb^k) - \frac{L - \beta L_{\Psi}}{2}\Vert \Vb_L(\xb^k) - \xb^k\Vert^2.
\end{array}{\!\!\!}
\end{equation*}
From this inequality, we can see that if $L\geq \beta L_{\Psi}$, then $F(\Vb_L(\xb^k)) \leq \mathcal{Q}_L(\Vb_L(\xb^k);\xb^k)$.
Hence, Step~\ref{step:A_1_GN_step} of Algorithm~\ref{alg:A1} terminates after a finite number of iterations.
\end{proof}

Let $\mathcal{L}_F(\alpha) =\{x \in\Omega ~\mid~ F(x) \leq \alpha \}$ be the level set of $F$ at $\alpha$.  
Now, we are ready to state the following theorem on global convergence of Algorithm~\ref{alg:A1}. 

\begin{theorem}\label{thm:main_theorem}
Let $\set{\xb^k}$ be the sequence generated by Algorithm~\ref{alg:A1}.
Then $\set{\xb^k}\subset \mathcal{L}_F(F(\xb^0))$ and 
\begin{equation}\label{eq:convergence_rate}
\min_{0\leq k \leq K}\Vert \Gb_{\beta L_{\Psi}}(\xb^k)\Vert^2 \leq \frac{2(\beta L_{\Psi})^2}{ L_{\min} ({K+1})} \left[ F(\xb^0) - F^{\star}\right],
\end{equation}
where $F^{\star} := \inf_{\xb\in\Omega} F(\xb) > -\infty$. Moreover, we also obtain
\begin{equation}\label{eq:thm31_est3}
\lim_{k\to\infty}\norms{\xb^{k+1}-\xb^k} = 0,~~\text{and}~~\lim_{k\to\infty}\Vert \Gb_{\beta L_{\Psi}}(\xb^k)\Vert = 0,
\end{equation}
and the set of limit points $\hat{\mathcal{S}}^{*}$ of the sequence $\{\xb^k\}_{k\geq 0}$ is connected.
If this sequence is bounded (in particular, if $\mathcal{L}_F(F(\xb^0))$ is bounded) then every limit point is a stationary point of \eqref{eq:NLP}.
Moreover, if the set of limit points $\hat{\mathcal{S}}^{*}$  is finite, then the sequence $\set{\xb^k}$ converges to a stationary point $\xb^{*} \in \mathcal{S}^{*}$ of \eqref{eq:NLP}.
If, in addition, $\xb^{\ast}$ is feasible to \eqref{eq:constr_prob} and $\beta$ is sufficiently large, then $\xb^{*}$ is also a stationary point of \eqref{eq:constr_prob}.
\end{theorem}
\begin{proof}
From Step~\ref{step:A1_x_update} of Algorithm \ref{alg:A1}, we have $\xb^{k+1} := \Vb_{L_k} (\xb^k)$ and $\set{\xb^k}\subset\Omega$. 
Using \eqref{eq:key_est1}, it is easy to obtain  $-\infty < F^{\star} \leq F(\xb^{k+1}) \leq F(\xb^k) \leq \cdots \leq F(\xb^0)$.  
This shows that $\set{\xb^k}\subset\mathcal{L}_F(F(\xb^0))$, and $\set{F(\xb^k)}$ is a decreasing sequence and bounded.
Hence, it has at least a  convergent subsequence. Moreover, from \eqref{eq:key_est1}, we also have
\begin{equation}\label{eq:thm1_est1}
F(\xb^{k+1}) \leq F(\xb^k) - \frac{L_{\min}}{2} r_{L_k}^2(\xb^k)   \leq F(\xb^k) - \frac{L_{\min}}{2} r_{\beta L_{\Psi}}^2(\xb^k).
\end{equation}
Summing up the inequality \eqref{eq:thm1_est1} from $k=0$ to $k=K$ and using $F(\xb^{k+1}) \geq F^{\star}$, we obtain
\begin{equation*} 
\frac{L_{\min}}{ 2 (\beta L_{\Psi})^2 } \sum_{k=0}^K\Vert \Gb_{\beta L_{\Psi}}(\xb^k)\Vert^2  = \frac{L_{\min}}{2}\sum_{k=0}^K r^2_{\beta L_{\Psi}}(\xb^k)   \leq  F(\xb^0) - F(\xb^{k+1})  \leq   F(\xb^0) - F^{\star}.
\end{equation*}
This implies
\begin{equation*}
\min_{0\leq k\leq K}\Vert \Gb_{\beta L_{\Psi}}(\xb^k)\Vert^2 \leq \frac{2 (\beta L_{\Psi})^2}{L_{\min} (K+1)} \left[F(\xb^0) - F^{\star}\right],
\end{equation*}
which  leads to \eqref{eq:convergence_rate}. Similarly, for any $N \geq 0$ one has
\begin{equation}\label{eq:thm1_est2}
F(\xb^k) - F(\xb^{k+N}) \geq \frac{L_{\min}}{2}\sum_{i=k}^{k+N-1} r_{L_k}^2(\xb^i) \geq \frac{L_{\min}}{2}\sum_{i=k}^{k+N-1} r^2_{\beta L_{\Psi}}(\xb^i).
\end{equation}
Note that the sequence $\set{F(\xb^k)}_{k\geq 0}$ has a convergent subsequence, thus  passing to the limit as $k \to \infty$ in \eqref{eq:thm1_est2} we obtain the first limit of \eqref{eq:thm31_est3}.
Since $\norms{\xb^{k+1} - \xb^k} = r_{L_k}(\xb^k) \geq r_{\beta L_{\Psi}}(\xb^k) = \frac{1}{\beta L_{\Psi}}\norms{\Gb_{\beta L_{\Psi}}(\xb^k)}$ due to Statement (b) of Lemma~\ref{le:stationary_point}, the first limit of  \eqref{eq:thm31_est3} also implies the second one. If the sequence $\set{\xb^k}_{k\geq 0}$ is bounded, by passing to the limit through a subsequence and combining with Lemma \ref{le:stationary_point}, we easily prove that every limit point is a stationary point of \eqref{eq:NLP}.
If the set of limit points $\hat{\mathcal{S}}^{*}$ is finite{, by} applying the result in \cite{Ostrowski1966}[Chapt. 28], we obtain the proof of the remaining conclusion.
\end{proof}

%

Theorem~\ref{thm:main_theorem} provides a global convergence result for Algorithm~\ref{alg:A1}.  
Moreover, our algorithm requires solving convex subproblems at each iteration, thus offering a great advantage over classical penalty-type schemes.  
Since the underlying problem is non-convex, the iterates of our algorithm may get trapped at points that may be infeasible for the original problem.  
That is, under the stated conditions, the iterate sequence $\set{\xb^k}$ may converge to a  local minimum (stationary) point $\xb^{*}$ of \eqref{eq:NLP}. Since $\xb^{*} \in\Omega$, if $\Psi(\xb^{*}) = 0$, then $\xb^{*}$ is also a local minimum  (stationary) point of the original problem \eqref{eq:constr_prob}. 

We can {sometimes overcome} this by combining the algorithm with a \textit{run-and-inspect procedure} \cite{CheSun:19}, whereby if $\xb^{*}$ violates $\Psi(\xb) =0$, then we restart the algorithm at a new starting point.  More precisely, we  add an  \textit{inspect} phase to our  existing algorithm that helps escape from non-feasible stationary points. 
In the inspection phase, if  $\Psi(\xb^{*}) \not = 0$ we sample a  point  around the current point and increase the parameter $\beta$.  
{Since we do not know any optimal Lagrange multiplier $\yb^{*}$ of \eqref{eq:constr_prob}, we cannot guarantee that $\beta > \|\yb^{*}\|_{\infty}$.
However, since it is expected that the multiplier $\yb_k$ of the subproblem \eqref{eq:GN_dir} converges to $\yb^{*}$, we can use $\yb_k$ to monitor the update of $\beta$ by guaranteeing that $\beta > \vert\yb_k\vert_{\infty}$.}
We have seen that such strategy performs well on a set of  realistic non-convex  AC-OPF problems. Nevertheless, in this paper, we do not have theoretical guarantee for this variant and leave this extension for future work.


\subsection{Local convergence analysis}\label{subsec:local_convergence_rate}
Let us study a special case of \eqref{eq:NLP} where Algorithm~\ref{alg:A1} has a local quadratic convergence rate. 
Our result relies on the following assumptions. First, for simplicity of our analysis, we assume that $\beta > 0$ is fixed and $L_k$ is also fixed at $L_k := L > 0$ for $k\geq 0$ in Algorithm~\ref{alg:A1}. 
Next, let $\xb^{\ast}$ be a stationary point of  \eqref{eq:NLP} such that
\begin{equation}\label{eq:non_degenerate}
\iprods{\nabla{f}(\xb^{\ast}), \xb - \xb^{\ast}} + \beta\vert \Psi'(\xb^{\ast})(\xb - \xb^{\ast})\vert \geq {\omega_{\min}}\norms{\xb - \xb^{\ast}}, ~~~\forall \xb \in\mathcal{N}(\xb^{\ast})\cap\Omega,
\end{equation}
where  ${\omega_{\min}} > 0$ is a given constant independent of $x$ and $\mathcal{N}(\xb^{\ast})$ is a neighborhood of $\xb^{\ast}$. The condition \eqref{eq:non_degenerate} is rather technical, but it holds in the following case. Let  us assume that the Jacobian $\Psi'(\xb^{\ast})$ of $\Psi$ at $\xb^{\ast}$ satisfies the following condition 
\begin{equation*}
\norms{\Psi'(\xb^{\ast})(\xb - \xb^{\ast}} \geq \sigma_{\min}(\Psi'(\xb^{\ast}))\norms{\xb - \xb^{\ast}}~~~\forall~\xb \in\mathcal{N}(\xb^{\ast})\cap\Omega,
\end{equation*}
where $\sigma_{\min}(\Psi'(\xb^{\ast})$ is the positive smallest singular value of $\Psi'(\xb^{*})$.
This condition is similar to the strong second-order sufficient optimality condition  \cite{Nocedal2006}, but only limited to the linear objective function.
In this case, we have $\vert \Psi'(\xb^{\ast})(\xb - \xb^{\ast})\vert \geq \Vert \Psi'(\xb^{\ast})(\xb - \xb^{\ast})\Vert \geq \sigma_{\min}(\Psi'(\xb^{\ast}))\norms{\xb - \xb^{\ast}}$.
Therefore, it leads to
\begin{equation*}
\begin{array}{ll}
\iprods{\nabla{f}(\xb^{\ast}), \xb - \xb^{\ast}} + \beta\vert \Psi'(\xb^{\ast})(\xb - \xb^{\ast})\vert &\geq \left( \beta \sigma_{\min}(\Psi'(\xb^{\ast}))  - \norms{\nabla{f}(\xb^{\ast})}\right)\norms{\xb - \xb^{\ast}}.
\end{array}
\end{equation*}
For $\beta > 0$ sufficiently large such that $\beta > \frac{\norms{\nabla{f}(\xb^{\ast})}}{\sigma_{\min}(\Psi'(\xb^{\ast}))}$, we have ${\omega_{\min}} := \beta \sigma_{\min}(\Psi'(\xb^{\ast}))  - \norms{\nabla{f}(\xb^{\ast})} > 0$, and the condition \eqref{eq:non_degenerate} holds. Now, we prove a local quadratic convergence of Algorithm~\ref{alg:A1} under assumption~\eqref{eq:non_degenerate}.

Note that a fast local convergence rate such as superlinear or quadratic is usually expected in Gauss-Newton methods, see, e.g., \cite[Theorem 2.4.1.]{kelley1999iterative}.
The following theorem shows that Algorithm~\ref{alg:A1} can also achieve a fast local quadratic convergence rate under a more restrictive condition \eqref{eq:non_degenerate}.

\begin{theorem}\label{th:local_quadratic_convergence}
Let $\set{\xb_k}$ be the sequence generated by Algorithm~\ref{alg:A1} such that it converges to a feasible stationary point $\xb^{\ast}$ of \eqref{eq:NLP}.
Assume further that $\xb^{\ast}$ satisfies condition \eqref{eq:non_degenerate} for some ${\omega_{\min}} > 0$.
Then, if $\xb_k\in\mathcal{L}_F(F(\xb_0))$ such that $\norms{\xb_k - \xb^{\ast}}  \leq \frac{2{\omega_{\min}}}{L + (\sqrt{n}+5)\beta L_{\Psi}}$, then $\xb_{k+1} \in\mathcal{L}_F(F(\xb_0))$ and 
\begin{equation}\label{eq:local_key_est1}
\norms{\xb_{k+1} - \xb^{\ast}} \leq \left[\frac{L + (\sqrt{n}+3)\beta L_{\Psi}}{2({\omega_{\min}} - L_{\Psi}\norms{\xb_k - \xb^{\ast}})}\right]\norms{\xb_k - \xb^{\ast}}^2.
\end{equation}
As a consequence, the sequence $\set{\xb_k}$ locally converges to $\xb^{\ast}$ at a quadratic rate.
\end{theorem}

\begin{proof}
Let  $d(\xb, \yb) := \Psi(\yb) - \Psi(\xb) - \Psi'(\xb)(\yb - \xb)$.
First, by the Lipschitz continuity of $\Psi'(\cdot)$, we have $\vert d(\xb, \yb) \vert \leq \sqrt{n}\Vert d(\xb, \yb)\Vert   \leq \frac{\sqrt{n}L_{\Psi}}{2}\norms{\yb - \xb}^2$.
In this case, from this estimate and \eqref{eq:subprob1_a}, for any $\xb, \yb\in \mathcal{L}_F(F(\xb_0))$, we can derive that 
\begin{equation*}
\begin{array}{ll}
\mathcal{Q}_L(\Vb_L(\xb); \xb) &:= \displaystyle\min_{\yb\in\Omega}\Big\{ \mathcal{Q}_L(\yb; \xb) := f(\yb) + \beta\vert\Psi(\xb) + \Psi'(\xb)(\yb - \xb)\vert + \frac{L}{2}\norms{\yb - \xb}^2\Big\} \vspace{1ex}\\
&= \displaystyle\min_{\yb\in\Omega}\Big\{ f(\yb) + \beta\vert \Psi(\yb) - d(\xb, \yb) \vert + \frac{L}{2}\norms{\yb - \xb}^2 \Big\} \vspace{1ex}\\
&\leq \displaystyle\min_{\yb\in\Omega}\Big\{ f(\yb) + \beta\vert \Psi(\yb))\vert + \beta\vert d(\xb, \yb) \vert + \frac{L}{2}\norms{\yb - \xb}^2 \Big\} \vspace{1ex}\\ 
&\leq \displaystyle\min_{\yb\in\Omega}\Big\{ f(\yb) + \beta\vert \Psi(\yb))\vert +  \frac{(L + \sqrt{n}\beta L_{\Psi})}{2}\norms{\yb - \xb}^2 \Big\} \vspace{1ex}\\
&\leq \displaystyle\min_{\yb\in\Omega}\Big\{ F(\yb)  +  \frac{(L + \sqrt{n}\beta L_{\Psi})}{2}\norms{\yb - \xb}^2 \Big\}.
\end{array}
\end{equation*}
This estimate together with $\Psi(\xb^{\ast}) = 0$ and $\xb = \xb_k$ imply that 
\begin{equation*} 
\mathcal{Q}_L(\xb_{k+1}; \xb_k) \equiv \mathcal{Q}_L(\Vb_L(\xb_k); \xb_k) \leq F(\xb^{\ast}) +  \frac{(L + \sqrt{n}\beta L_{\Psi})}{2}\norms{\xb_k - \xb^{\ast}}^2 = f^{\star}  +  \frac{(L + \sqrt{n}\beta L_{\Psi})}{2}\norms{\xb_k - \xb^{\ast}}^2.
\end{equation*}
Moreover, by $\mu_f$-convexity of $f$, we have $f(\xb_{k+1}) - f^{\star} \geq \iprods{\nabla{f}(\xb^{\ast}), \xb_{k+1} - \xb^{\ast}} + \frac{\mu_f}{2}\norms{\xb_{k+1} - \xb^{\ast}}^2$.
Using these last two estimates and the definition of $\mathcal{Q}_L$, we can show that
\begin{equation*}
\begin{array}{ll}
\frac{(L + \sqrt{n}\beta L_{\Psi})}{2}\norms{\xb_k - \xb^{\ast}}^2 &\geq f(\xb_{k+1}) - f^{\star} + \beta\vert \Psi(x_k) + \Psi'(\xb_k)(\xb_{k+1} - \xb_k)\vert + \frac{L}{2}\norms{\xb_{k+1} - \xb_k}^2 \vspace{1ex}\\
 &\geq   \frac{\mu_f}{2}\norms{\xb_{k+1} - \xb^{\ast}}^2 + \iprods{\nabla{f}(\xb^{\ast}), \xb_{k+1}) - \xb^{\ast}} + \beta\vert \Psi'(\xb_k)(\xb_{k+1} - \xb^{\ast}) \vspace{1ex}\\
 & + {~} \Psi(\xb_k) - \Psi(\xb^{\ast}) - \Psi'(\xb^{\ast})(\xb_k - \xb^{\ast})  + (\Psi'(\xb_k) - \Psi'(\xb^{\ast}))(\xb_{k+1} - \xb_k) \vert \vspace{1ex}\\
 &\geq   \frac{\mu_f}{2}\norms{\xb_{k+1} - \xb^{\ast}}^2 + \iprods{\nabla{f}(\xb^{\ast}), \xb_{k+1}) - \xb^{\ast}} + \beta\vert \Psi'(\xb^{\ast})(\xb_{k+1} - \xb^{\ast})\vert \vspace{1ex}\\
 & - \frac{\beta L_{\Psi}}{2}\Vert \xb_k - \xb^{\ast}\Vert^2 - \beta L_{\Psi}\norms{\xb_k - \xb^{\ast}}\norms{\xb_{k+1} - \xb_k}.
 \end{array}
\end{equation*}
Finally, using condition \eqref{eq:non_degenerate}, we obtain from the last estimate that
\begin{equation*}
\frac{(L + \sqrt{n}\beta L_{\Psi})}{2}\norms{\xb_k - \xb^{\ast}}^2  \geq \left({\omega_{\min}} - \beta L_{\Psi}\norms{\xb_k - \xb^{\ast}}\right)\norms{\xb_{k+1} - \xb^{\ast}} - \frac{3\beta L_{\Psi}}{2}\norms{\xb_k - \xb^{\ast}}^2.
\end{equation*}
Rearranging this inequality given that $\norms{\xb_k - \xb^{\ast}} < \frac{{\omega_{\min}}}{\beta L_{\Psi}}$, we obtain \eqref{eq:local_key_est1}. From \eqref{eq:local_key_est1}, we can see that if $\norms{\xb_k - \xb^{\ast}} \leq \frac{2{\omega_{\min}}}{L + (\sqrt{n} + 5)\beta L_{\Psi}}$, then $\norms{\xb_k - \xb^{\ast}} < \frac{{\omega_{\min}}}{\beta L_{\Psi}}$.
Moreover, $\norms{\xb_{k+1} - \xb^{\ast}} \leq \norms{\xb_k - \xb^{\ast}}$.
Hence, if $\xb_k \in\mathcal{L}_F(F(\xb_0))$, then $\xb_{k+1} \in\mathcal{L}_F(F(\xb_0))$.
The last statement is a direct consequence of  \eqref{eq:local_key_est1}.
\end{proof}

\beforesec 
\section{A Case Study: The Optimal Power Flow Problem}\label{sec:modeling}
\aftersec
	
In this section, we present the optimal power flow problem and its reformulation in a form that obeys the structure presented in the previous section. We then perform numerical experiments to validate the GN algorithm and we compare performance with IPOPT.

\subsection{Problem settings}

\paragraph{Original AC-OPF.} 
Consider a directed electric power network with a set of nodes ${\cal B}$ and a set of branches $\mathcal{L}$. The network consists of a set $\cal G$ of generators, with $\mathcal{G}_i$ denoting the set of generators at bus $i$. Denote $Y = G + jB$ as the system admittance matrix{, $G$ being the conductance and $B$ the susceptance ($j^2=-1$) \cite{Taylor2015}}. The decision variables of AC-OPF are the voltage magnitudes $\boldsymbol{v} \in \mathbb{R}^{|{\cal B}|}$, phase angles $\boldsymbol{\theta} \in  \mathbb{R}^{|{\cal B}|}$, and
the real and reactive power outputs of generators, which are denoted as $\boldsymbol{p} \in  \mathbb{R}^{|{\cal G}|}$ and $\boldsymbol{q} \in  \mathbb{R}^{|{\cal G}|}$. We will consider a fixed real and reactive power demand at every node $i$, which we denote as $P_i^d$ and $Q_i^d$, respectively. The constraints of the AC-OPF problem can be described by equations \eqref{eq:PBalance}-\eqref{eq:BoxCons}, see \cite{Kocuk2016SOCP}.
\begin{align}
& \sum_{j \in \mathcal{G}_i} p_j - P_i^d - G_{ii}v_i^2 - \sum_{(i,j)\in \mathcal{L}} v_iv_j \Big( G_{ij}  \cos(\theta_i-\theta_j) + B_{ij} \sin(\theta_i-\theta_j) \Big) \nonumber\\[-1mm]
&-\sum_{(j,i)\in \mathcal{L}} v_jv_i \Big( G_{ji}  \cos(\theta_j-\theta_i)+ B_{ji} \sin(\theta_j-\theta_i) \Big)=0 && \forall i\in \mathcal{B}, \label{eq:PBalance} \\
& \sum_{j \in \mathcal{G}_i} q_j - Q_i^d + B_{ii}v_i^2 - \sum_{(i,j)\in\mathcal{L}} v_iv_j \Big( G_{ij}  \sin(\theta_i-\theta_j) -  B_{ij} \cos(\theta_i-\theta_j) \Big) \nonumber\\[-1mm]
& - \sum_{(j,i)\in\mathcal{L}} v_jv_i \Big( G_{ji}  \sin(\theta_j-\theta_i)  -  B_{ji} \cos(\theta_j-\theta_i) \Big)=0  && \forall i\in\mathcal{B} \label{eq:QBalance},\\
& \left(-G_{ii}v_i^2 + G_{ij}v_iv_j\cos(\theta_i-\theta_j) + B_{ij}v_iv_j\sin(\theta_i-\theta_j)\right)^2\nonumber\\[-1mm]
& + \left(B_{ii}v_i^2 - B_{ij}v_iv_j\cos(\theta_i-\theta_j) + G_{ij}v_iv_j\sin(\theta_i-\theta_j)\right)^2 \leq S_{ij}^2 && \forall  (i,j) \in\mathcal{L}, \label{eq:LineLim1}\\[1mm]
& \left(-G_{jj}v_j^2 + G_{ji}v_jv_i\cos(\theta_j-\theta_i) + B_{ji}v_jv_i\sin(\theta_j-\theta_i)\right)^2 \nonumber\\
& + \left(B_{jj}v_j^2 - B_{ji}v_jv_i\cos(\theta_j-\theta_i) + G_{ji}v_jv_i\sin(\theta_j-\theta_i)\right)^2 \leq S_{ij}^2 && \forall  (i,j) \in\mathcal{L}, \label{eq:LineLim2} \vspace{1.25ex}\\
&\qquad \boldsymbol{\underline{V}} \leq \boldsymbol{v} \leq \boldsymbol{\overline{V}}, \qquad \boldsymbol{\underline{P}} \leq \boldsymbol{p} \leq \boldsymbol{\overline{P}}, \qquad \boldsymbol{\underline{Q}} \leq \boldsymbol{q} \leq \boldsymbol{\overline{Q}}, \qquad {\underline{\theta}_{ij}} \leq \theta_i-\theta_j \leq {\overline{\theta}_{ij}} && \forall (i,j)\in \mathcal{L}\label{eq:BoxCons}.
\end{align}
Constraints \eqref{eq:PBalance} and \eqref{eq:QBalance} correspond to the real and reactive power balance equations of node $i$. Constraints \eqref{eq:LineLim1} and \eqref{eq:LineLim2} impose complex power flow limits on each line, which we indicate by a parameter matrix $\boldsymbol{S}$. 
Constraints \eqref{eq:BoxCons} impose bounds on voltage magnitudes (indicated by parameter vectors $\boldsymbol{\underline{V}}$ and $\boldsymbol{\overline{V}}$), bounds on real power magnitudes (indicated by parameter vectors $\boldsymbol{\underline{P}}$ and $\boldsymbol{\overline{P}}$), bounds on reactive power magnitudes (indicated by parameter vectors $\boldsymbol{\underline{Q}}$ and $\boldsymbol{\overline{Q}}$), and  bounds on voltage phase angles differences (indicated by ${\underline{\theta}_{ij}}$ and ${\overline{\theta}_{ij}}$ for each line $(i,j) \in \mathcal{L}$). Note that the power balance equality constraints  \eqref{eq:PBalance}, \eqref{eq:QBalance} as well as the inequality constraints \eqref{eq:LineLim1}, \eqref{eq:LineLim2} are non-convex with respect to $\boldsymbol{v}$ and $\boldsymbol{\theta}$.



We will consider {the} objective of minimizing real power generation costs. Hence, we consider a convex quadratic objective function $f$:

\[ f(\boldsymbol{p}) =  \boldsymbol{p}^\top\text{diag}({\boldsymbol{C_2}})\boldsymbol{p} + \boldsymbol{C_1}^\top \boldsymbol{p}, \]

where $\boldsymbol{C_2} \geq 0$ and $\boldsymbol{C_1}$ are given coefficients of the cost function. The AC OPF problem then reads as the following non-convex optimization problem:
\begin{equation*}
{\cal P}_{opt}: \qquad   \min_{(\boldsymbol{V}, \boldsymbol{\theta}, \boldsymbol{p}, \boldsymbol{q})} f(\boldsymbol{p}) \quad \text{subject to} \;\;   \eqref{eq:PBalance}-\eqref{eq:BoxCons}.
\end{equation*}

\beforesubsec
\paragraph{Quadratic reformulation.}
The starting point of our proposed GN method for solving problem ${\cal P}_{opt}$ is the quadratic reformulation of AC-OPF \cite{Exposito1999}. In this reformulation, we use a new set of variables $c_{ij}$ and $s_{ij}$ for 
replacing the voltage magnitudes and angles {$v_i  \angle \theta_i = v_i e^{j\theta_i} = v_i(\cos(\theta_i)+j\sin(\theta_i))$}. These new variables are defined for all $i \in  {\cal B}$ and $(i,j) \in \mathcal{L}$ as:
\begin{align}
& c_{ii} = v_i^2, \;  c_{ij} = v_i v_j \cos(\theta_i-\theta_j) \label{eq:DefC},\\
& s_{ii} = 0, \; s_{ij} = -v_i v_j \sin(\theta_i-\theta_j) \label{eq:DefS},
\end{align}
where we will denote the vectors $\boldsymbol{c}$ and $\boldsymbol{s}$ as the collection of the $c_{ij}$ and $s_{ij}$ variables, respectively.

Assuming $\underline{{\theta}}_{ij} > -\frac{\pi}{2}$ and $\overline{{\theta}}_{ij} < \frac{\pi}{2}, \ \forall (i,j) \in \mathcal{L}$ (which is common practice), the mapping from $(\boldsymbol{v}, \boldsymbol{\theta})$ to $(\boldsymbol{c}, \boldsymbol{s})$ defined by \eqref{eq:DefC}, \eqref{eq:DefS} can be inverted as follows:
\begin{eqnarray}
& & v_i = \sqrt{c_{ii}}, \quad \theta_{i} - \theta_j = \arctan\left(- \frac{s_{ij}}{c_{ij}}\right) \nonumber,
\end{eqnarray}
thereby defining a bijection in $(\boldsymbol{c}, \boldsymbol{s})$ and $(\boldsymbol{v}, \boldsymbol{\theta})$.

The set of $(\boldsymbol{c}, \boldsymbol{s})$ and $(\boldsymbol{v}, \boldsymbol{\theta})$ that define this bijection is further equivalent to the following set of non-linear non-convex constraints in $(\boldsymbol{c}, \boldsymbol{s}, \boldsymbol{v}, \boldsymbol{\theta})$, see \cite{Kocuk2016SOCP}:
\vspace{-1ex}
\begin{align}
&c_{ij}^2  + s_{ij}^2 =  c_{ii} c_{jj}
&& \forall (i,j) \in \mathcal{L} \label{eq:QuadCons} \\
& \sin(\theta_i-\theta_j)c_{ij} + \cos(\theta_i-\theta_j)s_{ij} = 0 && \forall (i,j)\in\mathcal{L}\label{eq:TanCons} \\
& s_{ii}  = 0 && \forall i \in\mathcal{B}. \notag
\vspace{-1ex}
\end{align}
Now, we will substitute the voltage magnitude variables into the problem ${\cal P}_{opt}$, and consider the problem on the variables $(\boldsymbol{c}, \boldsymbol{s}, \boldsymbol{\theta})$. 
This reformulation has been commonly employed in the literature in order to arrive at an SOCP (Second-Order Cone Programming) relaxation of the problem \cite{Exposito1999, Jabr2008}. Through numerical experiments, we demonstrate that this reformulation results in highly effective starting points for our algorithm based on the SOCP relaxation of the AC-OPF. Moreover, the reformulation preserves the power balance constraints in linear form. Thus, the power balance constraints are not penalized in our scheme, which implies that they are respected at every iteration of the algorithm. For all these reasons, we pursue the quadratic reformulation of the present section, despite the fact that it requires the introduction of the new variables $\boldsymbol{c}$ and $\boldsymbol{s}$.

Concretely, the AC power balance constraints \eqref{eq:PBalance} and \eqref{eq:QBalance} are linear in $c_{ij}$ and $s_{ij}$: 
\begin{align}
&\sum_{j \in \mathcal{G}_i} p_j - P_i^d - G_{ii}c_{ii} - \sum_{(i,j) \in \mathcal{L}} (G_{ij}c_{ij} - B_{ij}s_{ij}) \
- \sum_{(j,i) \in \mathcal{L}} (G_{ji}c_{ji} - B_{ji}s_{ji})= 0 \qquad \forall i \in \mathcal{B} \label{eq:QuadPBalance},\\
&\sum_{j \in \mathcal{G}_i} q_j - Q_i^d + B_{ii}c_{ii} + \sum_{(i,j) \in \mathcal{L}} (B_{ij}c_{ij} + G_{ij}s_{ij})
 +\sum_{(j,i) \in \mathcal{L}} (B_{ji}c_{ji} + G_{ji}s_{ji}) =  0 \qquad \forall i \in  \mathcal{B} \label{eq:QuadQBalance},
\end{align}
Similarly, the power flow limit constraints \eqref{eq:LineLim1} and  \eqref{eq:LineLim2} are convex quadratic   in $c_{ij}$ and $s_{ij}$: 
\begin{align}
& (-G_{ii}c_{ii} + G_{ij}c_{ij} - B_{ij}s_{ij})^2
 + (B_{ii}c_{ii} - B_{ij}c_{ij} - G_{ij}s_{ij})^2 \leq S_{ij}^2 && \forall (i,j)\in \mathcal{L}, \label{eq:LineLimFrom}\\
&(-G_{jj}c_{jj} + G_{ji}c_{ij} + B_{ji}s_{ij})^2
+ (B_{jj}c_{jj} - B_{ji}c_{ij} + G_{ji}s_{ij})^2 \leq S_{ij}^2 && \forall (i,j)\in \mathcal{L}. \label{eq:LineLimTo}
\end{align}
The box constraints \eqref{eq:BoxCons} are reformulated as follows:
\begin{align}
&\underline{V}_i^2 \leq c_{ii} \leq \overline{V}_i^2 \quad \forall i\in \mathcal{B},
\ \boldsymbol{\underline{P}} \leq \boldsymbol{p} \leq \boldsymbol{\overline{P}}, \ \boldsymbol{\underline{Q}} \leq \boldsymbol{q} \leq \boldsymbol{\overline{Q}}, \ {\underline{\theta}_{ij}} \leq \theta_i-\theta_j \leq {\overline{\theta}_{ij}} \quad \forall (i,j)\in \mathcal{L}. \label{eq:QuadBoxCons}
\end{align}
As a result, an equivalent formulation for the AC-OPF model ${\cal P}_{opt}$ is:
\begin{equation}\label{eq:quad_form}
{\cal P}^{cs\theta}_{opt}: \   \min_{(\boldsymbol{p}, \boldsymbol{q},\boldsymbol{c},\boldsymbol{s},\boldsymbol{\theta})} f(\boldsymbol{p}) \quad \text{s.t.} \;\;   \eqref{eq:QuadCons}-\eqref{eq:QuadBoxCons}.
\end{equation}
With this reformulation, the {decision variables} are $\boldsymbol{x}=(\boldsymbol{p}, \boldsymbol{q}, \boldsymbol{c}, \boldsymbol{s}, \boldsymbol{\theta})$. 
Constraints \eqref{eq:QuadPBalance}-\eqref{eq:QuadBoxCons} define a convex set. 
%
We also have two non-convex equality constraints:
\begin{itemize}
\item Constraints \eqref{eq:QuadCons}: $\Psi^{ij}_q(\boldsymbol{c}, \boldsymbol{s}):=c_{ij}^2 + s_{ij}^2 - c_{ii} c_{jj} = 0, \ \forall (i,j) \in \mathcal{L}$. 
We will refer to them as \textit{quadratic} constraints. 

\item Constraints \eqref{eq:TanCons}: $\Psi^{ij}_t(\boldsymbol{c}, \boldsymbol{s},\boldsymbol{\theta}) := \sin(\theta_i-\theta_j)c_{ij} + \cos(\theta_i-\theta_j)s_{ij} = 0, \ \forall (i,j) \in \mathcal{L}$. 
We will refer to them as \textit{trigonometric} constraints. 
\end{itemize}

Now, AC-OPF is written in the format of \eqref{eq:constr_prob}, where $\boldsymbol{x} := (\boldsymbol{p},\boldsymbol{q},\boldsymbol{c},\boldsymbol{s},\boldsymbol{\theta})$, $f(\boldsymbol{x)} := f(\boldsymbol{p})$, $\Psi(\boldsymbol{x}):=( \Psi^{ij}_q(\boldsymbol{c},\boldsymbol{s}), \Psi^{ij}_t(\boldsymbol{c},\boldsymbol{s},\boldsymbol{\theta}))$ and:
\begin{align}\label{eq:Omega_set}
\begin{array}{ll}
\Omega :=& \{ \boldsymbol{x} \in \mathbb{R}^d \mid \boldsymbol{x} \mbox{ satisfies } \eqref{eq:QuadPBalance}-\eqref{eq:QuadBoxCons}\}
\end{array}
\end{align}

Moreover, since $\Psi$ is the collection of \eqref{eq:QuadCons} and \eqref{eq:TanCons}, we show in the next lemma  that it is differentiable and that its Jacobian is  Lipschitz continuous. 

\begin{lemma}\label{le:Lipschitz}
For the AC-OPF problem, $\Psi$ defined by \eqref{eq:QuadCons}--\eqref{eq:TanCons}, is smooth, and its Jacobian $\Psi'$ is Lipschitz continuous with a Lipschitz constant $L_{\Psi}$, i.e. $\Vert \Psi'(\xb) - \Psi'(\hat{\xb})\Vert \leq L_{\Psi}\Vert \xb - \hat{\xb}\Vert$ for all $\xb,\hat{\xb}\in\Omega$, where
\begin{equation}\label{eq:Lipschitz_const_of_Psi}
L_{\Psi} := \max\set{2, \big(1 + 2\max\{ \overline{V}_i^2 \mid i\in\mathcal{B} \}\big)^{1/2}} < +\infty.
\end{equation}
\end{lemma}

The proof of Lemma \ref{le:Lipschitz} is shown in Appendix \ref{app:proof_lemma}. 

{We can now derive AC-OPF subproblems for GN. As a reminder, the subproblems are in the format of (8). Since (30), (31) and (34) are linear constraints and (32) and (33) are second-order cone constraints, $\Omega$, in the AC-OPF case, is a set of second-order cone constraints. By applying the classical substitution of the norm-1 terms in the objective, the objective of (8) is (convex) quadratic. We typically refer to an SOCP as an optimization problem with second-order cone constraints and a linear objective. As a consequence, the AC-OPF subproblem is a Quadratically Constrained Program (QCP).} 
We are now in a position to apply GN to AC-OPF.

\beforesubsec
\subsection{A practical implementation of the GN algorithm for AC-OPF}
In this section, the goal is to optimize the settings of the GN method. We will demonstrate that the choice of $\beta$ and $L$ is crucial. 
This will allow us to derive a practical version of the GN algorithm, which we compare to IPOPT. 
\paragraph{Stopping criteria.} 
We terminate Algorithm~\ref{alg:A1} in two occasions, which have been validated through experimental results:
\begin{itemize}
\item If the maximum number of iterations $k_{\max}:=100$ has been reached.
\item If the quadratic and trigonometric constraints are satisfied with a tolerance of $\epsilon_2$, where $\epsilon_2 :=1e^{-5}$.
Concretely, we stop Algorithm~\ref{alg:A1} if 
$$\max(\Vert\Psi_q(c^k,s^k)\Vert_\infty,\Vert\Psi_t(c^k,s^k,\theta^k)\Vert_\infty) \leq \epsilon_2.$$
\end{itemize}
If the difference $\Vert x^{k+1}-x^{k}\Vert_\infty < \epsilon_1$ ($\epsilon_1:=1e^{-6}$ in the numerical experiment), then Algorithm~\ref{alg:A1} has reached an approximate stationary point of the exact penalized formulation \eqref{eq:NLP} (see Lemma~\ref{le:stationary_point}(a)). In this case, the last iterate might not be feasible for $\mathcal{P}_{opt}^{cs\theta}$. We then use the run-and-inspect strategy \cite{CheSun:19}: the last iterate becomes the starting point of GN and $\beta$ is doubled. 

\paragraph{SOCP relaxation for initialization.}
As mentioned previously, the quadratic formulation is also used to derive the SOCP relaxation. In the SOCP relaxation, the angles $\boldsymbol{\theta}$ are not modeled. Consequently, the trigonometric constraints \eqref{eq:TanCons} are removed. Moreover, the non-convex constraints \eqref{eq:QuadCons} are relaxed:
\begin{align}
c_{ij}^2+s_{ij}^2 \leq c_{ii}c_{jj}, \quad \forall (i,j)\in\mathcal{L}.\label{eq:SOC}
\end{align}
Then, $\mathcal{P}_{socp}$ is defined such that:
\begin{equation*}
{\cal P}_{socp}: \   \min_{(\boldsymbol{p}, \boldsymbol{q},\boldsymbol{c},\boldsymbol{s})} f(\boldsymbol{p}) \quad \text{s.t.} \;\;   \eqref{eq:SOC}, \ \eqref{eq:QuadPBalance}-\eqref{eq:QuadBoxCons}.
\end{equation*}
Solving this relaxation will provide a partial initial point $(\boldsymbol{p}^0,\boldsymbol{q}^0,\boldsymbol{c}^0,\boldsymbol{s}^0)$. In order to improve this initial point, we also compute angles $\boldsymbol{\theta}^0$ by solving the following optimization problem, where the goal is to minimize the error on the trigonometric constraints:
\begin{alignat*}{4}
&\boldsymbol{\theta}^0=\arg\min_{\boldsymbol{\theta}} \ && \sum_{(i,j)\in\mathcal{L}}\left((\theta_i-\theta_j)-\mbox{arctan}(\frac{-s_{ij}^0}{c_{ij}^0})\right)^2
&\text{ subject to} \quad && {\underline{\theta}_{ij}} \leq \theta_i-\theta_j \leq {\overline{\theta}_{ij}} \quad \forall (i,j)\in \mathcal{L}.
\end{alignat*}

This initialization is used for the experimental results on the MATPOWER benchmark.

\paragraph{Parameter tuning strategies.}
The convergence theory presented above does not require tuning the $\beta$ and $L$ parameters. In practice, tuning is crucial for improving the performance of algorithms for constrained non-convex optimization, including Algorithm~\ref{alg:A1}. 
Several observations allow us to decrease the number of iterations that are required for convergence: 
(i) according to Proposition \ref{pro:penalty_stationary}, large values of $\beta$ ensure the equivalence between \eqref{eq:constr_prob} and \eqref{eq:NLP}; (ii) quadratic constraints and trigonometric constraints scale up differently; (iii) a careful updating of $L$ influences the number of times that subproblem \eqref{eq:subprob1_a} is solved.
These observations guide a detailed experimental investigation concerning the choices of $\beta$ and $L$ parameters, which is discussed in Appendix \ref{app:tuning_strategies}.

\paragraph{Accelerating the computation of subproblem solutions, and warmstart.}
The most challenging constraints in the subproblems of GN are the line limit constraints \eqref{eq:LineLimFrom} and \eqref{eq:LineLimTo}. All the constraints are linear except for \eqref{eq:LineLimFrom} and \eqref{eq:LineLimTo} which are quadratic convex. From experimental observations, it is very uncommon that the entire network is congested and that all line limits are binding. One way to benefit from this observation is the following: assuming $x^k$ is the $k^\text{th}$ iterate of the algorithm and $x^k$ is binding for a set $\mathcal{L}^k \subseteq \mathcal{L}$, a next guess $\hat{x}^{k+1}$ will first be computed by only enforcing \eqref{eq:LineLimFrom} and \eqref{eq:LineLimTo} for $(i,j) \in \mathcal{L}^k$. This approach of adding lazy network constraints to the model is applied both in theory \cite{Aravena2018} as well as in practice by system operators.

Moreover, we observe that the subproblems \eqref{eq:subprob1_a} are based on the same formulation, and only differ by slight changes of certain parameters along the iterations. 
This motivates us to warm-start the subproblem \eqref{eq:subprob1_a} with a previous primal-dual iterate. 
{In other words, we initialize the solver for solving the subproblem \eqref{eq:subprob1_a} at the $k$-th iteration at the final solution $\xb_{k-1}$ and its corresponding multiplier $\yb_{k-1}$ obtained from the previous iteration $k-1$.} 
Warm-starting is indeed a key step in iterative methods, including our GN scheme, and will be further analyzed  in Section \ref{sec:warm-start}.

\beforesubsec
\subsection{Numerical experiments}
In order to validate the proposed GN algorithm, our numerical experiments are conducted in 2 steps: first, we launch simulations on several test cases of a classical library (MATPOWER) and compare the GN algorithm with a state-of-the-art non-convex solver (IPOPT); second, we show the potential benefit of warm-start for our approach.

\subsubsection{Illustration on MATPOWER instances}
We use the MATPOWER \cite{ZimMur:11} library 
to have access to a wide range of AC-OPF test systems that have been investigated in the literature. We test our approach on instances whose size ranges between 1,354 and 25,000 nodes (\texttt{1354pegase} has 11,192 variables and 27,911 constraints while \texttt{ACTIVS25k} has 186,021 variables and 431,222 constraints). 

We benchmark our approach against IPOPT, a non-linear solver based on the interior-point method. To do so, we make use of \texttt{PowerModels.jl} \cite{Coffrin2018}, a Julia package that can be used to solve AC-OPF instances of different libraries with different formulations. In order to make a fair comparison, we initialize GN and IPOPT using the SOCP solution. 
{
Since the subproblem \eqref{eq:GN_dir} has a non-smooth convex objective due to the $\ell_1$-norm penalty, we reformulate it equivalently to a quadratically constrained program (QCP). 
We solve this QCP with Gurobi \cite{Gurobi}. Note that the time needed to solve the SOCP solution is not reported for two reasons. First, we use the SOCP solution for both GN and IPOPT as a starting point to make a fair comparison, so it will not change the gap between the execution times of the two methods. Second, computing the SOCP solution is negligible compared to the methods tested (always less than 5\%).}



The results of our analysis are presented in Table \ref{table:results}. In Table \ref{table:results}, from left to right, we report the name of the test case. In the Gauss-Newton part, we report the number of iterations, the number of times run-an-inspect is executed, the objective value{, the maximum constraint violation} and the execution time (in seconds)
. In the IPOPT part, we report the objective value{, the maximum constraint violation} and the execution time (in seconds). The last column provides the gap between the GN solution and the IPOPT solution ($\text{Gap} = \frac{\text{Obj}_\text{IPOPT}-\text{Obj}_\text{GN}}{\text{Obj}_\text{IPOPT}}$). 

\begin{table*}[htp!]
\caption{Comparison of the GN algorithm against IPOPT with SOCP initialization.}
\label{table:results}
\begin{center}
\newcommand{\cellb}[1]{\textbf{#1}}
\begin{tabular}{|l||c|c|c|c|c||c|c|c||c|c|c|}
\hline
& \multicolumn{5}{|c||}{Gauss-Newton} & \multicolumn{3}{|c||}{IPOPT} &\\
\hline
Test Case &\# It & \# RI & Objective & {MCV} & Time 
& Objective & {MCV} & Time & Gap\\
 \hline
\texttt{1354pegase} & 13 & 0 & $7.407e^4$ & {$4.0e^{-6}$} & 15.1 s
 & ${{7.407e^4}}$ & {$1.3e^{-7}$} & \textbf{{6.00}} s& 0.00 \%\\
\hline
\texttt{1888rte} & 14  & 0 & $5.981e^4$& {$8.0e^{-6}$}  & \textbf{25.9} s
 & ${\mathbf{5.980e^4}}$& {$1.4e^{-6}$}  & {59.0} s& 0.02 \%\\
\hline
\texttt{1951rte}& 4  & 0 & ${8.174e^4}$& {$6.9e^{-6}$}  & \textbf{6.94} s
 & ${8.174e^4}$& {$3.9e^{-6}$}  & {7.67} s  & 0.00 \%\\
\hline
\texttt{ACTIVSg2000}& 5 & 0 & ${1.229e^6}$& {$7.7e^{-6}$}  & \textbf{8.39} s
 &  ${1.229e^6}$& {$3.6e^{-7}$}  & {14.8} s& 0.01 \%\\
\hline
\texttt{2383wp} & 19  & 2 & $1.868e^6$ & {$1.7e^{-6}$} & 51.0 s
  & ${1.868e^6}$& {$1.1e^{-9}$}  & \textbf{{21.6}} s & 0.02 \% \\
\hline
\texttt{2736sp}& 4  & 0 & $\mathbf{1.307e^6}$& {$9.0e^{-6}$}  & \textbf{10.9} s
 & ${1.308e^6}$& {$8.5e^{-10}$}  & {13.3}  s& -0.06 \% \\
\hline
\texttt{2737sop}& 3 & 0 & $\mathbf{7.767e^5}$& {$9.2e^{-6}$}  & \textbf{7.77}  s
 &  ${7.778e^5}$& {$6.9e^{-10}$}  & 9.98 s& -0.14 \%\\
\hline
\texttt{2746wop}& 17 & 2 & $1.208e^6$& {$6.0e^{-6}$}  & 38.1 s
  & ${1.208e^6}$& {$2.6e^{-9}$}  & \textbf{12.8} s& 0.00 \%\\
\hline
\texttt{2746wp}& 3  & 0 & $\mathbf{1.631e^6}$& {$7.7e^{-6}$}  & \textbf{6.89} s
 &  ${1.632e^6}$& {$5.9e^{-9}$}  & {16.0} s& -0.05 \%\\
\hline
\texttt{2848rte}& 17 & 2 & ${5.303e^4}$& {$7.2e^{-6}$}  & \textbf{52.0} s
  & $\mathbf{5.302e^4}$& {$3.4e^{-7}$}  & {74.3} s & 0.01 \%\\
\hline
\texttt{2868rte} & 4 & 0 & ${7.980e^4}$& {$6.2e^{-6}$}  &\textbf{ {9.11}} s
  & $\mathbf{7.979e^4}$& {$3.9e^{-6}$}  & 29.3 s& 0.00 \%\\
\hline
\texttt{2869pegase} & 12 & 0 & $1.340e^5$& {$9.8e^{-6}$}  & 41.9 s
 & ${1.340e^5}$& {$2.5e^{-8}$}  & \textbf{{15.5} } s& 0.00 \% \\
\hline
\texttt{3012wp} & 8 & 0 & $2.593e^6$& {$2.2e^{-6}$}  & 25.2 s
  & $\mathbf{2.592e^6}$& {$9.4e^{-10}$}  & \textbf{{18.8}} s& 0.03 \%\\
\hline
\texttt{3120sp} & 13 & 0 & $\mathbf{2.142e^6}$& {$8.1e^{-6}$}  & {41.2} s
  & ${2.143e^6}$& {$7.3e^{-9}$}  & \textbf{19.2} s& -0.04 \%\\
\hline
\texttt{3375wp} & 8 & 0 & $7.413e^6$& {$8.7e^{-6}$}  &{34.7} s 
 & $\mathbf{7.412e^6}$& {$3.0e^{-8}$}  & \textbf{21.3} s & 0.02 \%\\
\hline
\texttt{6468rte} & 19 & 2 & ${8.685e^4}$& {$9.4e^{-6}$}  & 187 s
  & $\mathbf{8.683e^4}$& {$6.6e^{-7}$}   & \textbf{103}  s & 0.02 \% \\
\hline
\texttt{6470rte} & 11  & 0 & ${9.835e^4}$& {$4.9e^{-6}$}  & \textbf{89.3} s 
  & $9.835e^4$& {$3.8e^{-6}$}  & 144 s& 0.01 \%\\
\hline
\texttt{6495rte}& 12  & 0& ${1.063e^5}$& {$8.4e^{-6}$}  & {180}  s
 & ${1.063e^5}$& {$4.6e^{-8}$}  & \textbf{52.4} s & 0.04 \% \\
\hline
\texttt{6515rte}& 18  & 0 & ${1.098e^5}$& {$3.2e^{-6}$}  & 204 s 
 & ${1.098e^5}$& {$5.0e^{-7}$}  & \textbf{86.9} s& 0.03 \%\\
\hline
\texttt{9241pegase} & 17 & 0 & ${3.167e^6}$& {$9.2e^{-6}$}  & 894 s
 & $\mathbf{3.159e^6}$& {$1.2e^{-6}$}  & \textbf{368} s& 0.25 \%\\
\hline
\texttt{ACTIVSg10k}& 6  & 0& ${2.488e^6}$& {$6.6e^{-6}$}  & {117} s
 & $\mathbf{2.486e^6}$& {$6.8e^{-8}$}  & \textbf{93.4} s& 0.10 \%\\
\hline
\texttt{13659pegase}& 19 & 0 & $3.885e^5$& {$8.5e^{-6}$}  & \textbf{137} s
 & $\mathbf{3.861e^5}$& {$4.4e^{-7}$}  & 685 s & 0.62 \% \\
\hline
\texttt{ACTIVSg25k}& 16 & 0 & ${6.033e^6}$& {$9.9e^{-6}$}  & {1,740} s 
 & $\mathbf{6.018e^6}$& {$7.8e^{-8}$}  & \textbf{544} s & 0.25 \%\\
\hline
\end{tabular}
\end{center}
\vspace{-4ex}
\end{table*}

The first notable observation is that GN finds a stationary point {(i.e. feasible)} of the original AC-OPF problem for all 23 test cases. The stationary point obtained by GN attains the same objective function value as the one returned by IPOPT for most instances (a difference of 0.05\% in the objective value may be attributed to numerical precision) and the proposed method outperforms IPOPT in some instances (e.g. \texttt{2737sop}).

Our experiments further demonstrate that the GN method consistently requires a small number of iterations (less than 20) in a wide range of instances. 
This is critically important to further accelerate the performance of our method if we appropriately exploit warm-start strategies and efficient solvers for the strongly convex subproblem.

In terms of computational time, the performance is shared between the two approaches. Nevertheless, some instances reveal limitations of the GN algorithm, compared to IPOPT (\texttt{6495rte}, \texttt{9241pegase} and \texttt{ACTIVSg25k} for example): when the solution of a subproblem becomes time-consuming because of the size of the subproblem, GN might require a larger execution time. We use Gurobi for solving subproblem \eqref{eq:subprob1_a}, because it is one of the most stable QCP solvers that are available. 

Unfortunately, Gurobi (and IPMs for QCPs in general) does not support warm-start, which would have significantly decreased the computational time. One alternative is to use an Alternating Direction Method of Multipliers (ADMM) solver that supports warm-start. {There is no additional cost in implementing warm-start. Indeed, we simply use the solution of the previous iteration as a starting point for the next one.} However, ADMM solvers are not mature enough to test large-scale problems. Implementing an efficient subsolver is out of the scope of this work, however we are able to analyze the effect of warm start on these solvers (and by extension to our proposed algorithm), which is the subject of the next section. 




\subsubsection{The effect of warm-starting strategy}
\label{sec:warm-start}
%

We consider using OSQP \cite{OSQP} as an ADMM solver. Since OSQP only solves quadratic programs (QPs), and since we are only aiming at illustrating the potential of warm-start, we will drop the line limit constraints \eqref{eq:LineLimFrom} and \eqref{eq:LineLimTo}, in order to satisfy the requirements of the solver. We consider small test cases (\texttt{39\_epri} and  \texttt{118\_ieee}) since we observed numerical instability for larger test cases.

The results are presented in Table \ref{tab:warmstart}. From left to right, 
\# It reports the number of GN iterations, \# ADMM {It} reports the total number of ADMM iterations {performed} during the GN execution, and Time reports the sum of OSQP solve times along the iterations in seconds.

\begin{table*}[htp!]
\caption{Results without and with warm-start on 2 small instances using OSQP solver.
}
\label{tab:warmstart}
\begin{center}
\newcommand{\cellb}[1]{\textbf{#1}}
\begin{tabular}{|l||c|c|c||c|c|c|c|c|c|c|c|}
\hline
& \multicolumn{3}{|c||}{No warm-start} & \multicolumn{3}{|c|}{With warm-start} \\
\hline
Test Case &\# It & \# ADMM It & Time & \# It & \# ADMM It & Time\\
 \hline
\texttt{39\_epri} & 4 & 114,735 & 4.65 s & 4 & 56,988 & 2.15 s\\
\hline
\texttt{118\_ieee} & 4  & 162,251  & 29.2 s & 4 & 65,832 & 11.6 s\\
\hline
\end{tabular}
\end{center}
\vspace{-2ex}
\end{table*}

For  both cases, we observe that warm start divides the total number of ADMM iterations as well as solve time by more than a factor of 2, and almost by a factor of 3 for \texttt{118\_ieee}.

We also examine each GN iteration individually, and highlight the impact of warm-start on the number of ADMM iterations in Figure \ref{fig:WSEvolution}. Note that we warm-start dual and primal variables only after iteration 1. Warm-start decreases substantially the number of ADMM iterations in two cases:
\begin{enumerate}
\item When $L$ is updated. Indeed, updating $L$ only results in slightly changing the objective function. One expects the previous iterates to provide a good warm-start. This is confirmed in Figure \ref{fig:WSEvolution}, where we observe that the number of ADMM iterations is divided by at least a factor of 2 every time $L$ is updated.
\item When the last iterates are computed. Intuitively, one does not expects iterates to change substantially when approaching the optimal solution. This intuition is confirmed by Figure \ref{fig:WSEvolution}. For the particular case of the last iterate, for \texttt{39\_epri} (resp. \texttt{118\_ieee}), the required number of ADMM iterations is less than 20\% (resp. 30\%).
\end{enumerate}
This investigation suggests that, with a mature ADMM solver, warm-starting is a promising feature for improving the performance of GN on large test cases.

\begin{figure*}[htp!]
\centering
\begin{minipage}{0.5\textwidth}
  \includegraphics[width=8cm,height=5.25cm]{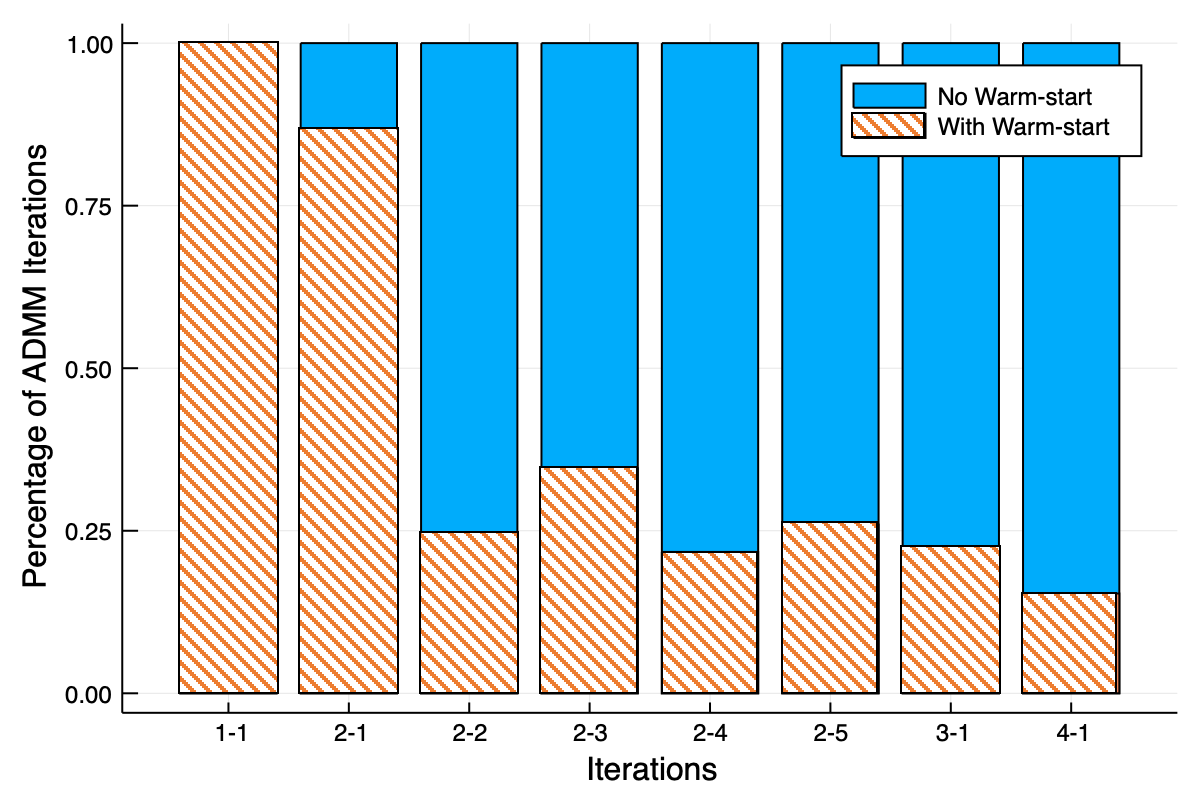}
\end{minipage}%
\begin{minipage}{.5\textwidth}
\includegraphics[width=8cm,height=5.25cm]{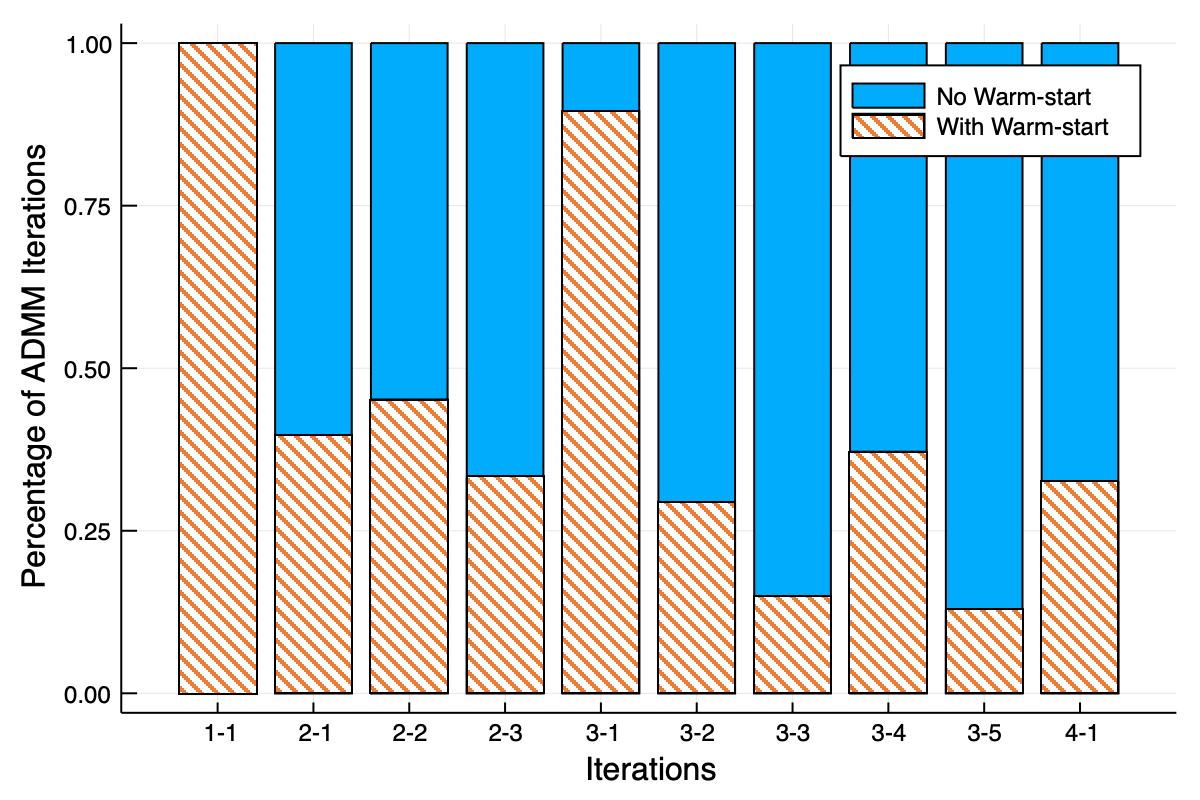}
\end{minipage}%
\caption{Evolution of the percentage of ADMM iterations along the iterations for \texttt{39\_epri} (left) and \texttt{118\_ieee} (right). `No Warm-start' always implies 100\% and the percentage of ADMM iterations `With Warm-start' is measured relatively to `No Warm-start'. GN iterations are shown on the $x$ axis in an $a-b$ format: $a$ is the actual GN iteration and $b$ represents the $b$th subproblem that had to be solved at iteration $a$ because of  an update of $L$.}
\label{fig:WSEvolution}
\vspace{-2ex}
\end{figure*}

\beforesubsec
{
\subsection{Optimization with bilinear matrix inequality constraint}\label{subsec:bmi_opt}
The goal of this subsection is to demonstrate that Algorithm~\ref{alg:A1} can solve a more general class of problems than classical methods such as IPMs or SQPs.
For this purpose, we consider the following optimization problem involving bilinear matrix inequality (BMI) constraints:
\begin{equation}\label{eq:bim_opt}
\left\{\begin{array}{ll}
\displaystyle\max_{t, F, P} & t \vspace{1ex}\\
\text{subject to}~&(A + BFC)^{\top}P + P(A + BFC) + 2t P \preceq 0, \ P \succeq 0, \ P = P^{\top},
\end{array}\right.
\end{equation}
%
where $t \in\mathbb{R}$, $F \in \mathbb{R}^{n_u\times n_y}$ and $P\in\mathbb{R}^{n\times n}$ are optimization variables, $A\in\mathbb{R}^{n\times n}$, $B\in\mathbb{R}^{n\times n_u}$, and $C\in\mathbb{R}^{n_y\times n}$ are given input matrices, and $(\cdot) \preceq 0$ (resp. $\succeq 0$) means that $(\cdot)$ is negative semidefinite (resp., positive semidefinite).
This problem arises from controller design, and is known as the spectral abscissa problem, see, e.g., \cite{Burke2002}.
\newline
If we introduce $S = -(A + BFC)^{\top}P + P(A + BFC) - 2t P$, $\xb = (S, P, F, t)$, and $\Omega := \big\{ \xb := (S, P, F, t) \mid S = S^{\top}, \ P = P^{\top}, \ S \succeq 0, \ P \succeq 0 \big\}$, then we can reformulate \eqref{eq:bim_opt} into \eqref{eq:constr_prob} with $f(\xb) := t$ and $g(\xb) := (A + BFC)^{\top}P + P(A + BFC) + 2t P + S = 0$.
Clearly, if we explicitly write $\Omega$ in the following form 
\begin{equation*}
\Omega = \big\{  (S, P, F, t) \mid S = S^{\top}, \ P = P^{\top}, \ \lambda_{\min}(S) \geq 0, \ \lambda_{\min}(P) \geq 0 \big\},
\end{equation*}
then we obtain a nonsmooth problem, where $\lambda_{\min}(\cdot)$ is the smallest eigenvalue of $(\cdot)$, which is nonsmooth \cite{Burke2002}.
In this case, IPMs and SQPs are not applicable for solving \eqref{eq:bim_opt}.
However, Algorithm~\ref{alg:A1} can still solve \eqref{eq:bim_opt} and its theoretical guarantees are preserved. 
\newline
We test our method by using the following data:
$$A=\begin{bmatrix}
-2.45 & -0.90 & +1.53 & -1.26& +1.76\\
   -0.12 &-0.44 &-0.01 &+0.69 &+0.90\\
     +2.07& -1.20& -1.14& +2.04& -0.76\\
   -0.59 &+0.07& +2.91& -4.63 &-1.15\\
   -0.74 &-0.23 &-1.19& -0.06 &-2.52
\end{bmatrix}
\quad B=\begin{bmatrix}
+0.81& -0.79& +0.00& +0.00& -0.95\\
   -0.34 &-0.50& +0.06 &+0.22 &+0.92\\
   -1.32 &+1.55 &-1.22& -0.77& -1.14\\
   -2.11 &+0.32 &+0.00 &-0.83& +0.59\\
   +0.31 &-0.19 &-1.09& +0.00& +0.00
\end{bmatrix}   $$
$$C=\begin{bmatrix}
+0.00 &+0.00& +0.16& +0.00& -1.78\\
   +1.23 &-0.38 &+0.75 &-0.38 &+0.00\\
   +0.46 &+0.00& -0.05& +0.00& +0.00\\
   +0.00 &-0.12 &+0.23 &-0.12 &+1.14
   \end{bmatrix}$$
As starting point, we use:
$$t^0=\lambda_{\min}((A+A^\top)/2), \ P^0 = S^0 = I_5, \ F^0 = 0_{5,4},$$
and we choose $\beta= 1,000, L_0=1$.
These choices are motivated by the fact that a large value of $\beta$ ensures convergence, see Proposition \ref{pro:penalty_stationary}. We apply the same update strategy for $L$ as in the AC-OPF implementation.\\
Our algorithm converges in 14 iterations to a feasible local optimum. The constraints are satisfied with a tolerance of $1e^{-6}$.
}

\aftersec

\beforesec 
\section{Conclusion}
\aftersec

We propose a novel Gauss-Newton algorithm for solving a general  class of optimization problems with non-convex constraints.  We utilize an exact non-smooth penalty reformulation of the original problem and suggest an iterative scheme for solving this penalized problem which relies on the non-squared Gauss-Newton method. The subproblems of our proposed scheme are strongly convex programs, which can be efficiently solved by numerous third-party convex optimization solvers. We establish a best-known global  convergence rate for our method to a stationary point. Under more restrictive conditions, we derive a local quadratic convergence rate for our scheme.

We apply our proposed approach to solve the optimal power flow problem, which is a fundamental and ubiquitous problem in power systems engineering. We apply our proposed algorithm to a reformulation of the AC-OPF, and we propose numerous strategies for tuning our proposed Gauss-Newton scheme, initializing the algorithm, and warm-starting the resolution of the subproblems that are treated by our proposed method. We perform extensive numerical experiments on a large set of instances from the MATPOWER library, and demonstrate the competitive performance of our method to IPOPT, which is a state of the art non-linear non-convex solver. This analysis validates the theoretical analysis of our proposed GN scheme, and proves its effectiveness in practical applications. Future work aims at broadening the scope of such applications beyond power systems.

\section*{Acknowledgments}
I. Mezghani and A. Papavasiliou acknowledge the financial support of ENGIE-Electrabel. I. Necoara would like to acknowledge the support from  the Executive Agency for Higher Education, Research and Innovation Funding (UEFISCDI), Romania,  PNIII-P4-PCE-2016-0731, project ScaleFreeNet, no. 39/2017.
The work of Quoc Tran-Dinh was partly supported by NSF (DMS-1619884).

\bibliographystyle{apalike}
\bibliography{all_refs}

\clearpage

\section*{Appendix}
\appendix
This appendix provides the full proof of technical results in the main text, the detailed implementation of our algorithm, and additional numerical experiments.

\section{The proof of Lemma \ref{le:Lipschitz}}
\label{app:proof_lemma}
\begin{proof}
From the definition of $\Psi$, it consists of two parts: quadratic forms in $(c_{ij}, s_{ij})$ and trigonometric and linear forms in $\theta_{ij} := \theta_i - \theta_j$ and $(c_{ij}, s_{ij})$, respectively.
We can write it as $\Psi = [\Psi^q, \Psi^t]$.
Each function in $\Psi^q$ has the form $c_{ij}^2 + s_{ij}^2 - c_{ii}c_{jj}$, as shown by \eqref{eq:QuadCons}, and each function in $\Psi^t$ has the form $\sin(\theta_{ij})c_{ij} + \cos(\theta_{ij})s_{ij}$, as shown in \eqref{eq:TanCons}.
We can show that the second derivative of each component of $\Psi^q$ w.r.t. $(c_{ii}, c_{jj}, c_{ij}, s_{ij})$ and of $\Psi^t$ w.r.t. $(c_{ij}, s_{ij}, \theta_{ij})$, respectively is
\begin{align*}
&\nabla^2\Psi^q(\xb) = \begin{bmatrix}0 & -1 & 0 & 0 \\ -1 & 0 & 0 & 0 \\ 0 & 0 & 2 & 0\\ 0 & 0 & 0 & 2\end{bmatrix},~~~\text{and}\\
&\nabla^2\Psi^t(\xb) = \small{\begin{bmatrix}0 & 0 & \cos(\theta_{ij}) \\ 0 & 0 & -\sin(\theta_{ij}) \\ \cos(\theta_{ij}) & - \sin(\theta_{ij}) & -c_{ij}\sin(\theta_{ij}) - s_{ij}\cos(\theta_{ij})\end{bmatrix}}.
\end{align*}
The second derivative $\nabla^2\Psi^q(\xb)$ is constant.
Hence, the maximum eigenvalue of $\nabla^2\Psi^q(\xb)$ is \\$\lambda_{\max}(\nabla^2\Psi^q(\xb)) = 2$.
For any $\ub := [u_1, u_2, u_3]\in\mathbb{R}^3$, we can easily estimate that
\begin{equation*}
\begin{array}{ll}
\ub^{\top}\nabla^2{\Psi}^t(\xb)\ub
&= u_1u_3[- \sin(\theta_{ij}) + \cos(\theta_{ij})] + u_2u_3[\cos(\theta_{ij})\\ &\quad - \sin(\theta_{ij})] + [-c_{ij}\sin(\theta_{ij}) - s_{ij}\cos(\theta_{ij})]u_3^2 \vspace{1ex}\\
&\leq \frac{1}{2}(u_1^2 + u_3^2) + \frac{1}{2}(u_2^2 + u_3^2) + (|c_{ij}| + |s_{ij}|)u_3^2 \vspace{1ex}\\
&\leq (1 + |c_{ij}| + |s_{ij}|)( u_1^2 + u_2^2 + u_3^2) \vspace{1ex}\\
&= (1 + |c_{ij}| + |s_{ij}|)\Vert u\Vert^2.
\end{array}
\end{equation*}
Therefore, $\lambda_{\max}\left( \nabla^2\Psi^t(\xb)\right) =  1 + \vert c_{ij}\vert + \vert s_{ij}\vert \leq  1 + 2\max\set{V^2_i \mid i \in\mathcal{B}}$, where the last inequality follows from \eqref{eq:BoxCons}.
Consequently,  
\begin{equation*}
L_{\Psi} := \max\set{2, \big(1 + 2\max\{ \overline{V}_i^2 \mid i\in\mathcal{B} \}\big)^{1/2}} < +\infty,
\end{equation*}
which is \eqref{eq:Lipschitz_const_of_Psi}.
\end{proof}

\section{Investigation on parameter tuning strategies}
\label{app:tuning_strategies}
In order to explain how the parameters of the GN algorithm should be tuned, we first express the subproblem at iteration $k$ for the special case of AC-OPF.
\paragraph{Subproblem \eqref{eq:subprob1_a}  in the context of AC-OPF.}
We define the following functions in order to keep our notation compact: 
\begin{align*}
&\Phi_{q}(\boldsymbol{c},\boldsymbol{s},\boldsymbol{dc},\boldsymbol{ds})=\sum_{(i,j)\in\mathcal{L}}\bigl\lvert \Psi_q^{ij}(\boldsymbol{c},\boldsymbol{s}) + {\Psi_q^{ij}}'(\boldsymbol{c},\boldsymbol{s})(\boldsymbol{dc},\boldsymbol{ds})^\top \bigr\rvert,\\[5pt]
&\Phi_{t}(\boldsymbol{c},\boldsymbol{s},\boldsymbol{\theta},\boldsymbol{dc},\boldsymbol{ds},\boldsymbol{d\theta}) = \sum_{(i,j)\in\mathcal{L}}\bigl\lvert \Psi_t^{ij}(\boldsymbol{c},\boldsymbol{s},\boldsymbol{\theta})+ {\Psi_t^{ij}}'(\boldsymbol{c},\boldsymbol{s},\boldsymbol{\theta})(\boldsymbol{dc},\boldsymbol{ds},\boldsymbol{d\theta})^\top\bigr\rvert.
\end{align*}
Note that we are facing two types of constraints (\textit{quadratic} and \textit{trigonometric}) and that they may attain different relative scales for different instances, as we have observed in our numerical experiments. 
We will define different penalty terms depending on the type of constraint, which will increase the flexibility of our implementation, while respecting the theoretical assumptions that we present in the main body of the paper. 
To this end, we denote by $\beta_q$ (resp. $\beta_t$) the $\beta$ penalty parameter associated with the quadratic (resp. trigonometric) constraints. The trigonometric constraints depend on $\boldsymbol{c}$, $\boldsymbol{s}$ and $\boldsymbol{\theta}$, whereas the quadratic ones only depend on $\boldsymbol{c}$ and $\boldsymbol{s}$. Therefore, we define two separate regularization parameters: $L_{cs}$ and $L_\theta$ (we drop the $k$ index from now on for notational simplicity).

The GN subproblem at iteration $k$, corresponding to \eqref{eq:GN_dir}, is  convex and has the form:
\vspace{-0.5ex}
\begin{alignat*}{3}
\mathcal{P}^k_{sub}: \quad & \min_{(\boldsymbol{y},\boldsymbol{d})} \quad && f(\boldsymbol{p}) + \beta_q\Phi_{q}(\boldsymbol{c}^k,\boldsymbol{s}^k,\boldsymbol{dc},\boldsymbol{ds})
 + \beta_t\Phi_{t}(\boldsymbol{c}^k,\boldsymbol{s}^k,\boldsymbol{\theta}^k,\boldsymbol{dc},\boldsymbol{ds},\boldsymbol{d\theta})\\&&&
  +\frac{L_{cs}}{2}\Vert (\boldsymbol{dc},\boldsymbol{ds})\Vert^2+\frac{L_{\theta}}{2}\Vert \boldsymbol{d\theta}\Vert^2\nonumber\\
& s.t. &&\boldsymbol{y}=(\boldsymbol{p}, \boldsymbol{q}), \ \boldsymbol{d}=(\boldsymbol{dc}, \boldsymbol{ds}, \boldsymbol{d\theta})\\
&&& (\boldsymbol{p},\boldsymbol{q},\boldsymbol{c}^k+\boldsymbol{dc},\boldsymbol{s}^k+\boldsymbol{ds},\boldsymbol{\theta}^k+\boldsymbol{d\theta}) \in \Omega.
\end{alignat*}
The optimal solution of $\mathcal{P}^k_{sub}$, which we denote by $(\boldsymbol{y}^\star,\boldsymbol{d}^\star)$, provides the next iterate $\boldsymbol{x}^{k+1}=(\boldsymbol{p}^\star,\boldsymbol{q}^\star, \boldsymbol{c}^k+\boldsymbol{dc}^\star, \boldsymbol{s}^k+\boldsymbol{ds}^\star, \boldsymbol{\theta}^k+\boldsymbol{d\theta}^\star)$.

\beforesubsec
\subsection{Joint values of parameters: $\beta_q=\beta_t$ and $L_{cs}=L_\theta$}
\aftersubsec
In our first set of tests, we implement the basic variant of Algorithm~\ref{alg:A1} by retaining the configuration of parameters from our theoretical results. 
Since we only aim at validating the algorithm, we focus on the three instances that have less than 2,000 nodes: \texttt{1354pegase}, \texttt{1888rte}, and \texttt{1951pegase}. 

For the penalty parameters, we first choose the same value $\beta$ for both quadratic and  trigonometric constraints as $\beta=\beta^q=\beta^t$.
We perform tests with three choices of $\beta$, based on the number of lines in the test case: $\beta=|\mathcal{L}|/100, |\mathcal{L}|/10$, or $|\mathcal{L}|$.
For the regularization parameter $L$, we also choose the same value $L =L_{cs}=L_\theta$. Note that all the experiments are initialized from the SOCP solution.
We consider three different strategies when applying Algorithm \ref{alg:A1}:
\begin{itemize}
\item \textit{Fixed strategy}: we fix $L$ at the upper bound $L_{\psi}$, which is computed in \eqref{eq:Lipschitz_const_of_Psi}.
\item \textit{Bisection update}: At each iteration $k$ of Algorithm~\ref{alg:A1}, we choose $L_{\min} = 1$ and initialize $L:=1$. If we do not satisfy the line-search condition $F(V_{L}(\boldsymbol{x}^k)) \leq \mathcal{Q}_{L}(V_{L}(\boldsymbol{x}^k); \boldsymbol{x}^k)$, we apply a bisection in the interval $[L_{\min},\beta L_{\psi}]$ until we satisfy this condition.
\item \textit{Geometric-$\mu$ update}: 
At each iteration $k$, we also choose $L_{\min} = 1$ and initialize $L:=1$. We then update $L\leftarrow \mu\cdot L$ for $\mu > 1$ in order to guarantee that $F(V_{L}(\xb^k)) \leq \mathcal{Q}_{L}(V_{L}(\xb^k); \xb^k)$. 
In our experiments, we use $\mu:=2$.
\end{itemize}
The results of our first test are presented in Table \ref{table:results_fixed}. The three columns for each strategy present the number of iterations, and the maximum violation of the quadratic (resp. trigonometric) constraint for each strategy on the three problem instances.
In terms of number of iterations, we can observe that fixing $L$ results in a poor performance and tends to increase the number of iterations as well as the final violation of the constraints. Given this poor performance, we do not consider the \textit{Fixed} strategy for the remainder of the numerical experiments. Also, using $\beta=|\mathcal{L}|$ produces satisfactory results, and increasing the value of $\beta$ supports convergence, therefore we choose this order of magnitude for the value of $\beta$.

\begin{table}[hpt!]
\caption{Performance results of the Gauss-Newton algorithm with different strategies for $\beta$ and $L$.}
\label{table:results_fixed}
\vspace{-1ex}
\begin{center}
\begin{tabular}{|l|c|c|c|c|c|c|c|c|c|c|}
\hline
&  & \multicolumn{3}{|c|}{Fixed} & \multicolumn{3}{|c|}{Bisection} & \multicolumn{3}{|c|}{Geometric-2}\\
\hline
Test Case & $\beta$ & \# It & MVQ & MVT & \# It & MVQ & MVT & \# It & MVQ & MVT\\
\hline
\texttt{1354pegase} & $|\mathcal{L}|/100$ & 100 & $6e^{-3}$ & $4e^{-6}$ & 100 & $6e^{-3}$ & $4e^{-6}$ & 100 & $7e^{-3}$ & $4e^{-6}$\\
& $|\mathcal{L}|/10$ & 91 & $2e^{-6}$ & $6e^{-9}$ & 3 & $5e^{-6}$ & $2e^{-7}$ & 3 & $5e^{-6}$ & $4e^{-6}$\\
& $|\mathcal{L}|$ & 78 & $2e^{-6}$ & $1e^{-8}$ & 3 & $8e^{-9}$ & $4e^{-9}$ & 3 & $6e^{-6}$ & $5e^{-6}$\\
\hline
\texttt{1888rte} & $|\mathcal{L}|/100$& 100 & $1e^{-2}$ & $8e^{-8}$ & 21 & $1e^{-2}$ & $9e^{-8}$ & 100 & $7e^{-3}$ & $8e^{-7}$\\
& $|\mathcal{L}|/10$& 100 & $8e^{-3}$ & $5e^{-8}$ & 11 & $7e^{-3}$ & $7e^{-8}$ & 100 & $5e^{-3}$ & $3e^{-8}$\\
& $|\mathcal{L}|$ & 100 & $8e^{-3}$ & $5e^{-8}$ & 15 & $5e^{-3}$ & $2e^{-8}$ & 100 & $3e^{-3}$ & $8e^{-8}$\\
\hline
\texttt{1951rte}& $|\mathcal{L}|/100$ & 86 & $4e^{-6}$ & $4e^{-9}$ & 10 & $2e^{-3}$&  $5e^{-8}$& 39 & $3e^{-4}$ & $4e^{-8}$\\
& $|\mathcal{L}|/10$ & 66 & $6e^{-7}$ & $3e^{-8}$ & 4 & $2e^{-5}$ & $6e^{-8}$ & 3 & $8e^{-6}$ & $6e^{-6}$\\
& $|\mathcal{L}|$ & 61 & $1e^{-6}$ & $1e^{-7}$ & 4 & $5e^{-7}$ & $1e^{-7}$ & 3 & $3e^{-6}$ & $2e^{-6}$\\
\hline
\end{tabular}
\end{center}
\vspace{-2ex}
\end{table}

\textit{Bisection} and \textit{Geometric-2} exhibit a similar behavior: when the algorithm converges to a feasible point, both choices achieve converge in tens of iterations, depending on the test case and the choice of $\beta$. Nevertheless, when failing, the maximum violation of the quadratic constraint (MVQ) never reaches the desired tolerance of $1e^{-5}$. 
This behavior might suggest that we do not penalize sufficiently the quadratic constraint. One should notice that the quadratic and trigonometric constraints are linked: once the angles are fixed, $\boldsymbol{c}$ and $\boldsymbol{s}$, which are the variables that appear in the quadratic constraints, struggle to move from their current value in order to satisfy the quadratic constraints. 
This motivates us to consider two different values for $\beta$ and $L$: $\beta_q$ for the quadratic constraints and $L_{cs}$ for the associated variables $\boldsymbol{c}$ and $\boldsymbol{s}$; $\beta_t$ for the trigonometric constraints and $L_\theta$ for the angle variables 
$\boldsymbol{\theta}$.



\subsection{Adaptation and enhancement of Algorithm \ref{alg:A1} for AC-OPF}
\paragraph{Adapting Algorithm \ref{alg:A1} with individual choices of parameters: $\beta_q$, $\beta_t$ and $L_{cs}$, $L_\theta$.} Based on our observations, we choose different values for $\beta_q$ and $\beta_t$.
Since we empirically observe that the quadratic constraints are harder to satisfy than the trigonometric constraints, we consider two alternatives: $\beta_q=2\beta_t$ and $\beta_q=5\beta_t$. Furthermore, we set $L_{cs,\min}=\beta_q/ \beta_t$ and $L_{\theta,\min}=1$.
Note that the choice of different values for these parameters does not affect the theoretical guarantees of our algorithm, as long as they satisfy our given conditions.

Also, since we have different values $L_{cs}$ and $L_\theta$, we must adapt the condition under which $L_{cs}$ and $L_\theta$ are updated. From the theory, $L$ is updated (through the \textit{Bisection} or \textit{Geometric} strategy) if the following condition is not met:
\begin{align*}
&&&F(\boldsymbol{x}^{k+1}) \leq \mathcal{Q}_{L}(\boldsymbol{x}^{k+1}; \boldsymbol{x}^k)\\
&\Leftrightarrow 
&&\beta_q\sum_{(i,j) \in \mathcal{L}} |\Psi_q^{ij}(\boldsymbol{c}^{k+1},\boldsymbol{s}^{k+1})|
 + \beta_t\sum_{(i,j) \in \mathcal{L}} |\Psi_t^{ij}(\boldsymbol{c}^{k+1},\boldsymbol{s}^{k+1},\boldsymbol{\theta}^{k+1})|\\
&&&  \leq \beta_{q}\phi_{q}(\boldsymbol{c}^k,\boldsymbol{s}^k,\boldsymbol{dc}^*,\boldsymbol{ds}^*) + \beta_{t}\phi_{t}(\boldsymbol{c}^k,\boldsymbol{s}^k,\boldsymbol{\theta}^k,\boldsymbol{dc}^*,\boldsymbol{ds}^*,\boldsymbol{d\theta}^*)
 + \frac{L_{cs}}{2}\Vert (\boldsymbol{dc}^*,\boldsymbol{ds}^*)\Vert^2 + \frac{L_{\theta}}{2}\Vert(\boldsymbol{d\theta}^*)\Vert^2
\end{align*}
where $\boldsymbol{x}^{k+1}=\boldsymbol{x}^{k}+\boldsymbol{d}^*$.

We adapt this condition to the specific type of constraint. Concretely:
\begin{itemize}[leftmargin=*]
\item If \begin{equation}
\begin{aligned}
& \beta_t\sum\limits_{(i,j) \in \mathcal{L}} |\Psi_t^{ij}(\boldsymbol{c}^{k+1},\boldsymbol{s}^{k+1},\boldsymbol{\theta}^{k+1})|\\[-3mm] 
&\quad \leq \beta_{t}\Phi_{t}(\boldsymbol{c}^k,\boldsymbol{s}^k,\boldsymbol{\theta}^k,\boldsymbol{dc}^*,\boldsymbol{ds}^*,\boldsymbol{d\theta}^*)
+ \frac{L_{cs}}{2}||(\boldsymbol{dc}^*,\boldsymbol{ds}^*)^\top||^2 + \frac{L_{\theta}}{2}||(\boldsymbol{d\theta}^*)||^2 
\end{aligned}
\label{eq:conditionTan} 
\end{equation}
then update $L_{cs}$ and $L_\theta$.
\item If \eqref{eq:conditionTan} does not hold and \begin{equation}
\begin{aligned}&\beta_q\sum\limits_{(i,j) \in \mathcal{L}} |\Psi_q^{ij}(\boldsymbol{c}^{k+1},\boldsymbol{s}^{k+1})| 
\leq \beta_{q}\Phi_{q}(\boldsymbol{c}^k,\boldsymbol{s}^k,\boldsymbol{dc}^*,\boldsymbol{ds}^*)+ \frac{L_{cs}}{2}||(\boldsymbol{dc},\boldsymbol{ds})^\top||^2\end{aligned}
\label{eq:conditionQuad}
\end{equation}
then only update $L_{cs}$.
\end{itemize}

\paragraph{Refining the $L$ updates.}
We emphasize that, in Algorithm~\ref{alg:A1}, each time that the values of $L_{cs}$ and/or $L_{\theta}$ are updated, the subproblem $\mathcal{P}_{sub}^k$ is resolved. 
Therefore, an effective strategy for updating these parameters can lead to significant improvements in computational time.
We mitigate this heavy computational requirement by introducing resolution techniques that are guided by both our theoretical results and empirical observations. Concretely, we propose the following two improvements to the practical implementation of the algorithm, in order to limit the number of computationally expensive $L$ updates:
\begin{enumerate}
	\item Keeping the value of $L$ from one iteration to another is a better strategy than having $L:=L_{\min}$ at the beginning of each iteration. The natural justification for this is that we expect steps to decrease along the iterations. 	\item Checking conditions \eqref{eq:conditionTan} and \eqref{eq:conditionQuad} can require a large number of $L$ updates. Instead, we propose a verification of whether the violation of the constraints is decreasing before checking  \eqref{eq:conditionTan} and \eqref{eq:conditionQuad}, which is a less stringent requirement that still yields satisfactory results in terms of constraint violations. Concretely, at iteration $k$, we compute the $\ell_1$ and $\ell_\infty$ norms of $\Psi_q(\boldsymbol{c}^k,\boldsymbol{s}^k)$ and $\Psi_t(\boldsymbol{c}^k,\boldsymbol{s}^k,\boldsymbol{\theta}^k)$. If these quantities decrease from $k-1$ to $k$, we move to iteration $k+1$.
\end{enumerate}

By applying these two approaches, we obtain the results that are presented in Table \ref{table:results_diffbetadiffL_cutL}. In Table \ref{table:results_diffbetadiffL_cutL}, \# It provides the number of GN iterations,  \# L-It provides the number of additional subproblems solved because of an $L$-update (then \# It+\# L-It gives the total number of subproblems solved) and Obj records the objective value returned. From Table \ref{table:results_diffbetadiffL_cutL}, we observe that this new strategy leads to convergence for all three test cases. Even if Bisection seems to require less iterations, it also provides higher objective values than Geometric-2. Also, we do not observe notable differences between applying a factor of 2 or 5 on $\beta$, although a factor of $5$ decreases slightly the number of iterations. In our implementation, we employ the following settings: Geometric-2, $\beta_t=|\mathcal{L}|$ and $\beta_q=5\beta_t$.

\begin{table}[hpt!]
\caption{Performance behavior of the Gauss-Newton algorithm with different values and strategies for $\beta_q$, $\beta_t$, $L_{cs}$, $L_\theta$.}
\label{table:results_diffbetadiffL_cutL}
\begin{center}
\begin{tabular}{|l|cc|c|c|c|c|c|c|c|c|c|}
\hline
& & & \multicolumn{3}{|c|}{Bisection} & \multicolumn{3}{|c|}{Geometric-2}\\
\hline
Test Case & $\beta_t$  & $\beta_q$ & \# It & \# L-It & Obj & \# It & \# L-It & Obj\\
\hline
\texttt{1354pegase}
& $|\mathcal{L}|$ & $\times 2$ & 3 & 1 & $7.408e4$ & 14 & 6 &  $7.407e4$\\
& $|\mathcal{L}|$ & $\times 5$  & 4 & 1 & $7.408e4$ & 13 & 4 &  $7.407e4$\\
\hline
\texttt{1888rte} 
& $|\mathcal{L}|$ & $\times 2$&  8 & 1 & $5.982e4$ & 17 & 1 & $5.981e4$\\
& $|\mathcal{L}|$ & $\times 5$  & 11 & 1 & $5.982e4$ & 14 & 3 & $5.981e4$ \\
\hline
\texttt{1951rte}
& $|\mathcal{L}|$ & $\times 2$ & 4 & 1 & $8.174e4$ & 6 & 1 & $8.174e4$\\
& $|\mathcal{L}|$ & $\times 5$  & 4 & 0 & $8.174e4$ & 4 & 0 & $8.174e4$\\
\hline
\end{tabular}
\end{center}
\vspace{-2ex}
\end{table}

\end{document}